\documentclass[12pt]{article}
\usepackage{a4wide}
\usepackage{amsmath, amsthm}
\usepackage{amssymb}

\usepackage{color}
\usepackage[usenames,dvipsnames]{xcolor}

\makeatletter
\newtheorem*{rep@theorem}{\rep@title}

\makeatother

\renewcommand{\subset}{\subseteq}

\newif\ifdraft

\drafttrue% this line was commented out, which is why the comments weren't working
\newif\ifcolorcomments

\colorcommentstrue

\newcommand{\allowcomments}[4]{
\newcommand{#1}[1]{\ifdraft{\ifcolorcomments{\textcolor{#4}{##1 --#3}}\else{\textsl{##1 \ --#3}}\fi}\else{}\fi}
}

\allowcomments{\comdavid}{DS}{David}{green}
\allowcomments{\comsanju}{SV}{Sanju}{blue}

\font\tenmsy=msbm10 scaled 1200 \font\sevenmsy=msbm7 scaled 1200
\font\fivemsy=msbm5 scaled 1200
\newfam\msyfam
\textfont\msyfam=\tenmsy \scriptfont\msyfam=\sevenmsy
\scriptscriptfont\msyfam=\fivemsy

\newcommand{\R}{{\mathbb R}}

\newcommand{\Z}{{\mathbb Z}}
\newcommand{\N}{{\mathbb N}}

\newcommand{\cH}{{\cal H}}
\newcommand{\cM}{{\cal M}}
\newcommand{\hf}{{\cal H}^{f}}
\newcommand{\hs}{{\cal H}^{s}}

\newcommand{\p}{\psi}

\newcommand{\I}{{\rm I}}

\newcommand{\ep}{ \varepsilon }

\newtheorem{thdet}{{\bf Theorem A}\!\!}
\newtheorem{thdetprime}{{\bf Theorem A$'$}\!\!}
\newtheorem{thdeta}{{\bf Proposition A}\!\!}

\newtheorem{thbs}{Theorem (Peres-Solomyak)\!\!}
\newtheorem{thcs}{Theorem (Peres--Shmerkin)\!\!\!}
\newtheorem{thdel}{Theorem (Marstrand)\!\!}

\newtheorem{theorem}{Theorem}
\newtheorem{lemma}[theorem]{Lemma}
\newtheorem{corollary}[theorem]{Corollary}
\newtheorem{proposition}[theorem]{Proposition}
\newtheorem{definition}[theorem]{Definition}
\newtheorem{thKJa}{Theorem (Khintchine--Jarn\'{\i}k) \!\!\!\!}

\theoremstyle{remark}
\newtheorem{rem}{Remark}

\newcommand{\be}{\begin{eqnarray*}}
\newcommand{\ee}{\end{eqnarray*}}

\newcommand{\ve}{\varepsilon}

\newcommand{\proj}{{\rm{proj}}}

\begin{document}

\title{Marstrand's Theorem revisited: projecting sets of dimension zero }

\author{ Victor Beresnevich\footnote{Research partly supported by EPSRC grant EP/J018260/1 }   \\ {\small\sc (York) }  \and
Kenneth Falconer \\ {\small\sc (St Andrews) }  \\  { \ }
\and Sanju Velani\footnote{Research partly supported by EPSRC grant EP/J018260/1 } \\ {\small\sc (York)}
 \and
  Agamemnon Zafeiropoulos \footnote{Research  supported by EPSRC Doctoral and  Departmental Teaching Studentships} \\ {\small\sc (York)}   % \and ~ \hspace{9cm}
 \\ ~ \\  {\Large With an Appendix: The gap of uncertainty} ~ \vspace{-2ex}  \and ~ \hspace{19cm} { \ }  \\     ~\vspace{-3ex}  \\  \and   {\small David Simmons\footnote{Research supported by EPSRC grant EP/J018260/1 } } \\ {\footnotesize\sc (York) }   \and
{\small Han Yu} \\ {\footnotesize\sc (St Andrews) }  \and
 {\small  Agamemnon Zafeiropoulos } \\ {\footnotesize\sc (York)}    \\ ~ \\
}

\date{}

%\date{{\small\today}}

\maketitle

\begin{abstract}
\noindent We establish a refinement of Marstrand's projection theorem for Hausdorff dimension functions finer than the usual power functions, including an analogue of Marstrand's Theorem for logarithmic Hausdorff dimension.
\end{abstract}

\bigskip

%{\small\noindent\emph{Key words and phrases}: metric Diophantine
%approximation, ubiquity, logarithm law for geodesics
%
%\medskip


\section{Introduction}

 \subsection{Motivation} Given $0\leq \theta <\pi$,  let $L_{\theta}$ denote the line through the origin of $\R^2$ that forms an angle $\theta$ with the horizontal axis.  Let $\proj_{\theta}$ denote orthogonal  projection onto the line $L_{\theta}$ and $\dim A$ denote the Hausdorff dimension of a set  $A\subset \mathbb{R}^2$.    Then $ \proj_{\theta}$ is a Lipschitz mapping; indeed for all $\theta$,
 \begin{equation} \label{lippy}
|  \proj_{\theta}x    -   \proj_{\theta}y  |  \leq  |  x    -  y  |    \quad  \forall   x,y \in \mathbb{R}^2 \,  .
\end{equation}
 This together with the trivial fact that  $\proj_{\theta}A  $  is a subset of a line, implies
 \begin{equation} \label{triv}
 \dim \proj_{\theta} A  \leq \min\left\{1,  \dim A\right\}     \, ,
\end{equation}
see for example \cite[Proposition 3.3]{F}.
The famous projection theorem of Marstrand  \cite{Mar}, dating from 1954, tells us that equality holds in \eqref{triv} for almost almost all directions $\theta$ with respect to Lebesgue measure ${\mathcal L}$. Equivalently,   the exceptional values of $\theta \in [0,\pi)$  for which  the  inequality \eqref{triv} is strict,  form  a set of one-dimensional Lebesgue measure zero.

\begin{thdel} Let $A\subset \R^2$ be a Borel set.
~\vspace*{-1ex}
\begin{itemize}
\item[(i)] If $\dim A\leq 1$ then $\dim {\proj}_{\theta}A=\dim A$ for almost all $\theta \in[0,\pi)$.
    ~\vspace*{-1ex}
    \item[(ii)] If $\dim A>1$ then ${\mathcal L}(\proj_{\theta}A)>0$  for  almost all $\theta \in[0,\pi)$.
\end{itemize}
\end{thdel}

\noindent  Observe that   the measure  conclusion  of (ii)  is significantly stronger than the corresponding  dimension statement; it  trivially implies  that $\dim {\proj}_{\theta}A = 1 $ for almost all $\theta \in[0,\pi)$.    Marstrand's  proof  depends heavily on delicate and, in places, complicated geometric  and measure theoretic arguments.   Subsequently,  Kaufman \cite{K} gave a slick, two page, proof that made natural use of the potential theoretic characterization of Hausdorff dimension and Fourier transform methods.

Here we will be  concerned with the case when $\dim A\leq 1$.   Indeed, to motivate our investigation, consider the extreme situation  when $\dim A=~0$.  Then Marstrand's Theorem implies no more than the trivial statement that
$\dim \proj_{\theta} A = 0$ for all $\theta$.   Thus,  to obtain non-trivial information in such situations,  it is natural to ask whether a version of Marstrand's Theorem  remains valid for  finer notions of Hausdorff dimension. One consequence of Theorem~\ref{mainthm}, our main result, is the following  analogue of Marstrand's Theorem for logarithmic Hausdorff dimension, that is where the Hausdorff measures are defined with respect to dimension or gauge functions $(-\log r)^{-s}$ (for small $r$) and $s\geq 0$, and $\dim_{\rm  \log}A$ is the critical value of $s$ at which these measures jump from $\infty$ to $0$,  see \S\ref{mainlogsec} for the full definitions.

\vspace*{2ex}

\begin{theorem} \label{logResult} Let $A\subset  \R^2$ be a Borel set.  Then
  ~\vspace*{-1ex}
\begin{itemize}
\item[(i)] $ \dim_{\rm \log} {\proj}_{\theta}A\leq\dim_{\rm  \log} A$ for all $\theta \in[0,\pi)$,
    ~\vspace*{-1ex}
    \item[(ii)] $ \dim_{\rm \log} {\proj}_{\theta}A=\dim_{\rm  \log} A$ for ${\mathcal L}$-almost all $\theta \in[0,\pi)$.
\end{itemize}
\end{theorem}

\vspace*{2ex}

\begin{rem} \label{exceptremark}  By considering the size of sets  of  exceptional angles, see \S\ref{exceptangles}, we are further able to conclude that
\begin{equation} \label{exceptequation}
\dim_{\rm \log} \left\{  \theta \in[0,\pi)  :   \dim_{\rm \log} {\proj}_{\theta}A  < \dim_{\rm  \log} A  \right\}  \,  \le  \,    \dim_{\rm \log} A \, .
\end{equation}
Of course, the  interesting case is  when  $\dim_{\rm  \log} A $ is finite. Then,  by definition $\dim A = 0 $ and so  \eqref{exceptequation} is significantly stronger than Theorem \ref{logResult}.
%The latter simply implies that the `exceptional' set  on the right hand side of \eqref{exceptequation} is of one-dimensional Lebesgue measure zero.
\end{rem}

Before moving onto our main result, Theorem \ref{mainthm}, which is an analogue of Marstrand's Theorem for a general class of dimension functions, we consider an explicit class of  sets that has motivated our work and which illustrates and clarifies the need for statements  such as   Theorem \ref{logResult}.

\vspace*{2ex}

\noindent {\em  The motivating example.}  Let  $\p:\R^+\to\R^+$ be a  decreasing
 function. A point $(y_1,\dots,y_k)\in\R^k$ is called {\it
simultaneously $\psi$--approximable} if there are infinitely many
$q\in\N$ and $(p_1, \ldots,p_k) \in \Z^k$ such that
\begin{equation}
%\max_{1\le i\le k}
 \left|y_i - \frac{p_i}{q}\right|\ <\
\psi(q) \hspace{9mm}   1
\leq i \leq k \ . \label{1}
\end{equation}
 The set of simultaneously
$\p$--approximable points in $\I^k:=[0,1]^k$ will be denoted by
$W_k(\psi)$. For convenience,  we work  within the unit cube
$\I^k$ rather than $\R^k$  as it makes full measure results easier
to state and avoids ambiguity. This is not at all
restrictive as the set of simultaneously $\psi$-approximable
points is invariant under translations by integer vectors.   The following statement provides a beautiful and simple criterion for the `size'  of  $W_k(\psi)$  in terms of
Hausdorff measures with respect to a dimension function $f$, see  \S\ref{mainresultsec} for the full definition of these measures.

\begin{thKJa}
\label{mainkj} Let $\p:\R^+\to\R^+$ be  a  decreasing
function. Let $f$ be a dimension function such that   $r^{-k}
\, f(r) $ is monotonic. Then $$ \hf\left(W_k(\p)\right)=\left\{\begin{array}{cl} 0
& {\rm \ if} \;\;\; \sum_{q=1}^{\infty}  \; q^k \,
f\left(\p(q)\right)
 <\infty \; ,\\[2ex]
\hf(\I^k) & {\rm \ if} \;\;\; \sum_{q=1}^{\infty}  \; q^k \,
f\left(\p(q)\right) =  \infty           \; .
\end{array}\right.$$
\end{thKJa}

\medskip

\noindent This theorem unifies the fundamental  results of
Khintchine and Jarn\'{\i}k in the classical theory of metric
Diophantine approximation. Khintchine's Theorem (1924) corresponds
to the situation in which $f(r)=r^k$ when $\hf$ is equivalent to
$k$-dimensional Lebesgue measure.   Jarn\'{\i}k's Theorem (1931)
corresponds to the situation in which $r^{-1} \, f(r) \to\infty$
as $r\to0$ and $r^{-k} \, f(r) $ is decreasing in which case $\hf(\I^k) =
\infty$.  For  background and further details see \cite{Beresnevich-Bernik-Dodson-Velani-Roth, HDSV97} and references therein.

For all $\tau > 0$, let $\psi_\tau$ be the `approximating' function given by
$
\psi_\tau(q) :=   \exp(-q^\tau  )  \,  .
$
Then by definition, when $k=1$  the  corresponding set $ W_k(\p_\tau)$  is  a  subset of the set of Liouville numbers which is well-known to be of Hausdorff  dimension zero.  In fact  $ \dim W_k(\p_\tau)= 0$ for all positive integers $k$.  To see this, note that for any dimension function  $f_s(r) = r^s \ (s> 0)$,
$$
\sum_{q=1}^{\infty}  \ \; q^k \,
f_s\left(\p_\tau(q)\right)  \ =  \ \sum_{q=1}^{\infty}  \
 \exp\big(-(s \,  q^{\tau}  - k \log q )   \big)
\  <  \ \infty \,
$$
for all $\tau >0$ and $k \in \R$.   Hence, it   follows from the Khintchine-Jarn\'{\i}k Theorem  and the definition of Hausdorff dimension that $ \dim   W_k(\p_\tau) = 0 $ for all $\tau > 0 $ and $ k \ge 1 $.  The upshot of this is that by \eqref{triv}, for all $\theta \in[0,\pi)$
$$\dim \big(\proj_{\theta} \, W_2(\p_\tau) \big) = 0 $$
and Marstrand's Theorem is not particularly informative.  The problem is that the dimension functions $f_s$ given by  $f_s(r)=r^s$ are not delicate enough  to differentiate between sets of dimension zero.   Instead, for $s >0$  consider
the logarithmic dimension function $f_s$ given by $f_s(r)=
(-\log r)^{-s}$ for $0<r<1$.  Then, for $ \tau > 0 $ and $ k \ge 1$, it is easily verified that
$$
\sum_{q=1}^{\infty}  \ \; q^k \,
f_s\left(\p_\tau(q)\right)  \ =  \ \sum_{q=1}^{\infty}  \
 q^{- (\tau s  -   k )    }    \ \ \
 \left\{\begin{array}{cl} < \infty
& {\rm \ if} \;\;\; s  > s_0 \; \\[1ex]
= \infty  & {\rm \ if} \;\;\; s  < s_0          \;
\end{array}\right.  ,
 $$
where
$$
s_0 := \frac{k+1}{\tau} \ .
$$
It then  follows from the Khintchine-Jarn\'{\i}k Theorem  and the definition of  logarithmic Hausdorff dimension, see \S\ref{mainlogsec}, that $     \dim_{\rm \log}    W_k(\p_\tau) = s_0 $ for all $\tau > 0 $ and $ k \ge 1 $. In turn Theorem \ref{logResult} implies the non-trivial statement that for almost all $\theta \in[0,\pi)$,
$$
\dim_{\rm \log} \big({\proj}_{\theta} \   W_k(\p_\tau)   \big)= s_0   \, . $$

\subsection{The main result  \label{mainresultsec}}
We first recall the definition of  $f$-dimensional Hausdorff measure.  Let  $f: \R^{+} \rightarrow \R^{+}$  be a {\em dimension}  or {\em gauge function}, that is a function
that is increasing and
continuous  with $f(r)\to 0$ as $r\to 0 \, $.
Let $A$ be  a non--empty subset of $\R^n$. For $\rho > 0$,
let
$$
\cH^{f}_{\rho} (A) \, := \, \inf \Big\{ \sum_{i} f(|U_i|)  : A \subset \bigcup_{i} U_{i} ,\ |U_i| \leq\rho \Big\} \, ,
 $$
 where $|U|$ denotes the diameter of a set $U$ and the infimum is over countable covers $\{U_{i}\}$ of $A$ by sets of diameter at most $\rho$.  The {\em
Hausdorff $f$-measure}  of $A$ is defined by
$$ \cH^{f}
(A) := \lim_{ \rho \rightarrow 0} \cH^{f}_{\rho} (A)   \;.
$$
When  $f(r) = r^s$ ($s > 0$), the measure $ \hf$ is the usual {\em $s$-dimensional Hausdorff measure}\/ $\hs $.

%Let $A$ be  a non--empty subset of $\R^n$. For $\rho > 0$, a countable collection $
%\left\{B_{i} \right\} $ of balls in $\R^n$ with radius $r_i \leq
%\rho $ for each $i$ such that $A \subset \bigcup_{i} B_{i} $ is
%called a {\em $ \rho$-cover}\/ for $A$. Let
%$$
%\cH^{f}_{\rho} (A) \, := \, \inf \Big\{ \sum_{i} f(r_i)  : \{
%B_{i} \}  \mbox{\rm \ is\ a\  $\rho$-cover\ of\ } A \Big\} \, ,
% $$
 %where the infimum is over all $\rho$-covers.  The {\em
%Hausdorff $f$-measure}  of $A$ is defined by
%$$ \cH^{f}
%(A) := \lim_{ \rho \rightarrow 0} \cH^{f}_{\rho} (A)   \;.
%$$

We will also use centred Hausdorff measure. Here we consider covers by a countable collection of balls $\left\{B(x_i, r_i)\right\} $ of radii $r_i \leq \rho $ with centres in $A$. Thus, for $\rho>0$ we set
$$
\cH^{f}_{C,\rho} (A) \, := \, \inf \Big\{ \sum_{i} f(r_i)  :  A \subset \bigcup_{i} B(x_i, r_i),\ x_i \in A, \ r_i \leq\rho \Big\} \, ,
 $$
and define the {\em centred Hausdorff $f$-measure}  of $A$ by
$$ \cH_C^{f}
(A) := \lim_{ \rho \rightarrow 0} \cH^{f}_{C,\rho} (A)   \;.
$$
These two measures are equivalent, in the sense that for all $A \subset \R^n$
 \begin{equation}  \label{mesequiv}
 \cH_C^{f}(A) \leq  \cH^{f}(A) \leq m_n\, \cH_C^{f}(A),
\end{equation}
where $m_n$ depends only on $n$. This follows easily from the definitions, noting that every set $U$ that intersects $A$ is contained in a ball with centre in $A$ and diameter $|U|$, and that every ball $B\subset \R^n$ of radius $r$ is contained in a finite number $m_n$ of balls of radius $\frac12 r$, that is diameter $r$; in particular $m_2=7$.

Note that  $f$-Hausdorff measure only depends on $f(r)$  for $r\in [0,r_0]$ for arbitrarily small $r_0$, so changing the dimension function $f$ outside a neighbourhood of $0$ does not affect the measure.

The {\em Hausdorff dimension}  $\dim A$ of a set $A$ is
defined by $$ \dim  A \, := \, \inf \left\{ s : \cH^{s}
(A) =0 \right\}   \, =     \,   \, \sup \left\{ s : \cH^{s}
(A) =\infty \right\} .  $$
It follows from \eqref{mesequiv} that we get the same value for Hausdorff dimension if we replace $\cH^{s}$ by  $\cH_C^{s}$ in this definition. For further discussion of Hausdorff measures and dimensions, see \cite{F,Mat, Rog}.
%When $s$ is an integer $\hs$ is a constant multiple of $s$--dimensional Lebesgue measure on $\R^s$.

Defining Hausdorff measures for general dimension functions allows a more precise notion of dimension than just a numerical value. For example, a set $A$ may have Hausdorff dimension $s$ but with $ \cH^{s}(A) = 0$. However, it may be that $0 < \cH^{f}(A) <\infty$ where, say $f(r) = r^s\log(1/r)$,  in which case we think of $A$ having dimension `logarithmically smaller' than $s$. Introducing a partial order $\prec$ on the set of dimension functions by $f\prec g$ if $\lim_{r\to 0} g(r)/f(r) =0$, which implies that $ \cH^{g}(A) = 0$ whenever $ \cH^{f}(A)< \infty$, allows a much finer notion of dimension,  see \cite{Rog}. It is also worth noting that there are sets $A\subset \R^n$ for which there
is no dimension function $f$ such that $0 < \cH^{f}(A) <\infty$, see \cite{Dav}.

\medskip

In order to state our main theorem we need the notion of doubling.  A dimension function $f$ is said to be {\em doubling} if there exist constants $c>1 $ and $r_0 > 0$ such that
\begin{equation}  \label{double}
f(2r) \leq cf(r)    \quad  \forall \,  0< r < r_0 \, .
\end{equation}

\noindent The number $c$ is called a {\em doubling constant}. Note that   if $f$ is given by $f(r) = r^s$ ($s > 0$) then $f(2r) = 2^s f(r) $  and so $c = 2^s$ is a doubling constant  for $f$.

We are now in the position to state our main result.

%\begin{theorem}\label{mainthm} Let $A\subset \R^2$ be a Borel set.  Let $f$ and $g$ be dimension functions such that  \begin{equation} \label{(1)}
% -\int_{0}^{1}\!f(r) \ \mathrm{d \!}\left(\frac{1}{g(r)}\right) <\infty  \, .
%\end{equation}
%Furthermore, assume that $g$ is  doubling  with constant  $c<2$. If
%$\cH^f(A) =0$ and $\cH^g(A) = \infty $,  then $\cH^f(\proj_{\theta}A) =0$ and $\cH^g(\proj_{\theta}A) = \infty $ for  almost all $\theta\in[0,\pi)$.
%\end{theorem}

\begin{theorem}\label{mainthm} Let $A\subset \R^2$ be a Borel set.
~\vspace*{-2ex}
\begin{itemize}
 \item[(i)]  Let $f$ be a dimension function. Then $\cH^f(\proj_{\theta}A) \le \cH^f(A)$  for all $\theta \in[0,\pi)$. In particular if  $\cH^f(A) =0$ then $\cH^f(\proj_{\theta}A) = 0$ for all $\theta \in[0,\pi)$.
  \item[(ii)] Let $f$ be a dimension function such that  $\cH^f(A) > 0 $. Suppose  $g$ is a dimension function that is  doubling  with constant  $c<2$ and such that
  \begin{equation} \label{(1)}
 -\int_{0}^{1}\!f(r) \ \mathrm{d \!}\left(\frac{1}{g(r)}\right) <\infty  \, .
\end{equation}
 %and suppose $g$ is  doubling  with constant  $c<2$.
 Then,  $\cH^g(\proj_{\theta}A) = \cH^g(A) = \infty $ for  almost all $\theta\in[0,\pi)$.
\end{itemize}

\end{theorem}

\medskip

Several remarks are in order.

\medskip

\begin{rem}  \label{part1}
Part (i) of Theorem \ref{mainthm} is an immediate consequence of the Lipschitz condition \eqref{lippy}  and the definition of $\cH^f$, see \cite[Proposition 3.1]{F} where the case of $f(r) = r^s$ is given. Thus the main substance of the theorem is part (ii) when  $  \cH^f(A) >  0 $.

%The fact that $\proj_{\theta}$ satisfies the Lipschitz condition \eqref{lippy} leads immediately to a  measure version of the  dimension statement  \eqref{triv}.
%Let $A\subset \R^2$ be a Borel set and $f$ be a dimension function.  Then, for every $\theta \in[0,\pi) $,  \begin{equation}  \label{lippymeasure}
%  \cH^f(\proj_{\theta}A) \le \cH^f(A)       \, .
%\end{equation}

%\noindent This is essentially a consequence of \cite[Proposition 2.2]{F} -- the $s$-dimensional Hausdorff measure proof of \cite[Proposition 2.2]{F} can be easily adapted to  general dimension functions $f$ and yield \eqref{lippymeasure}.  The upshot of this remark is that Theorem~\ref{mainthm} can be viewed as an attempt to determine conditions under which  equality holds in \eqref{lippymeasure}.   Clearly,  this is so if $  \cH^f(A) = 0 $  and this is precisely the content of  part (i) of Theorem \ref{mainthm}.  Thus the main substance of the theorem is case (ii) when  $  \cH^f(A) >  0 $.
\end{rem}

\medskip

\begin{rem} \label{rem2} Note that the conclusion of (ii) remains true if the range of integration in  \eqref{(1)} is an interval $[0, r_0]$ for any $r_0>0$. Moreover, if $g$ is differentiable, or at least differentiable except at finitely many points, then
\begin{equation}\label{intcond}
- \int_{0}^{1}\!f(r) \ \mathrm{d \!}\left(\frac{1}{g(r)}\right) =  \int_{0}^{1}\!f(r) \frac{g'(r)}{g^2(r)} \ \mathrm{d} r   \ .
\end{equation}
In particular, if $f$ and $g$ are dimension functions satisfying \eqref{(1)} then
\begin{equation}  \label{sv1} \lim_{r\to 0} \frac{f(r)}{g(r)}=0.\end{equation}
For suppose not. Then there exists $a>0$ and a sequence $r_n\searrow 0$ such that $f(r_n)/g(r_n)\geq a$ for all $n$. Let $r_n'>r_n$ be the least number such that $g(r_n') = 2g(r_n)$; such an $r_n'$ exists by continuity and monotonicity of $g$ provided that the sequence is chosen taking $r_1$ sufficiently small. Then
$$
\int_{r_n}^{r_n'}\!f(r) \frac{g'(r)}{g^2(r)} \ \mathrm{d} r
\ \geq \ \int_{r_n}^{r_n'}\! \frac{f(r_n)}{2g(r_n)}  \frac{g'(r)}{g(r)}\ \mathrm{d} r
\ \geq \ {\textstyle \frac{1}{2}} a \log \bigg[\frac{g(r_n')}{g(r_n)}\bigg]
\ = \ {\textstyle \frac{1}{2}} a \log 2.
$$
Since $0<r_n <r_n' \to 0$, the integrals in \eqref{intcond} and \eqref{(1)} cannot be finite.

%Indeed, the limit in question exists because $f,g$ are continuous functions on $\R^{+}$ and  assume for the moment that it equal to $a >0$.  Then, there exists some $r_0 >0$  so that   $f(r)>\frac{a}{2} \, g(r)$ for all $0<r<r_0$ and it  follows that
%\begin{eqnarray*}
%-\int_{0}^{r_0}f(r) \ \mathrm{d \!}\left(\frac{1}{g(r)}\right) & > & \int_{0}^{r_0}\frac{a}{2} \ g(r) \ \mathrm{d \!}\left( - \frac{1}{g(r)}\right) \\%[1ex]
%&=& \frac{a}{2} \  \int_0^{r_0}1 \ \mathrm{d \!}\left( \ln g(r) \right)  \, =  \, \infty \, .  \\[1ex]\end{eqnarray*}
%This  contradicts the assumption that condition  \eqref{(1)} holds and so we must have that $a=0$.

When contemplating an extension of  Marstrand's theorem to general dimension functions, it is not unreasonable to suspect
a  statement along the lines of Theorem \ref{mainthm}  with condition \eqref{(1)}  replaced by \eqref{sv1}.     The latter condition is natural and it initially appears to avoid known examples (see for instance   \cite[9.2 Example]{Mat})  of $s$-sets $A  \subset  \R^2$  with $0< s\le 1$  for which  $\cH^s(\proj_{\theta}A) = 0 $ for  all $\theta\in[0,\pi)$.
%\comdavid{A slightly more detailed analysis of Martin and Mattila's original example shows that it in fact satisfies $\cH^g(A) = 0$ for some function $g$ satisfying \eqref{sv1}, without requiring any modifications to their construction. However, it is necessary to modify their construction as is done in the appendix to get an optimal value for $g$.}
As we shall see in the appendix  the  construction of these sets can be adapted to show that we can not in general replace condition \eqref{(1)}  by \eqref{sv1} in Theorem \ref{mainthm}.  The following statement is easily deduced from the theorem   proved in the appendix for codoubling dimension function.   A dimension function $f$ is said to be {\em codoubling} if there exist constants $c>1 $ and $r_0 > 0$ such that
\begin{equation}  \label{codoublesv}
f(2r) \geq c \, f(r)    \quad  \forall \,  0< r < r_0 \, .
\end{equation}

\vspace{1mm}

\begin{thdetprime}
Let $f,g$ be dimension functions such that $f$ is doubling with constant  $c\le 2$ and codoubling, and such that
\begin{equation} \label{fgrelation1}
g(r) \leq M\ f\left(r\log(r^{-1})\right)    \quad  \forall \,  0< r < r_0
\end{equation}
where $r_0 > 0 $ and $ M>0$ are  constants. Then there exists a set $A\subseteq \R^{2}$ with $0< \cH^{f}(A) < \infty $ but $\cH^g(\proj_{\theta}A) = 0$ for all $0\leq \theta < \pi$.
\end{thdetprime}

It is easily seen that  Theorem A$'$ is not a  converse to Theorem \ref{mainthm}. Namely, fix some $ \delta \in (0,1) $  and for $ s > 0$ consider dimension functions $g_{s}$   given by
\[
g_{s}(r) = r^{\delta} \log^s(r^{-1})  \, .
\]
Clearly, $g_s$ is doubling with constant $c < 2$.  Also, let $f$ be  given by $f(r)=r^{\delta}$.  Clearly, $f$ is doubling with constant $c \le 2$ and codoubling. Then for any  $s \leq \delta$,  condition \eqref{fgrelation1} is satisfied  and Theorem~A$'$   implies that there exists a set $A_{s}$ with positive $\delta $-dimensional Hausdorff measure such that $\cH^{g_{s}}(\proj_\theta A_{s}) = 0$ for all $0\leq \theta < \pi$. On the other hand, for  any $s \geq 1$, condition \eqref{(1)} is satisfied and Theorem \ref{mainthm} implies that there does not exist such a set $A_{s}$.  But for $s \in (\delta,1)$, we do not know whether such a set exists.  Nevertheless, Theorem A$'$  shows that  we can not in general replace condition \eqref{(1)}  by \eqref{sv1}  and thus there is a gap of uncertainty associated with Theorem~\ref{mainthm} where we do not know what happens (for further discussion see \S\ref{fincomments}).   This gap of uncertainty will be explicitly highlighted in \S\ref{inad} when we return to our motivating example.  The fact that condition \eqref{(1)}   shows up is very much a consequence of the approach taken to prove the theorem.   Concerning  this the reader is especially directed  towards Remark \ref{gap}  at the end of \S\ref{enca} in the proof.

\noindent  Not surprisingly, condition \eqref{fgrelation1} implies that  the integral
convergence condition \eqref{sv1} is violated -- see the appendix for the details; namely Remark \ref{showingit}.

%\comsanju{OLD VERSION It would be interesting to know whether or not it is  possible  to replace \eqref{(1)}   by \eqref{sv1} in the statement of the theorem.   The fact that condition \eqref{(1)}  shows up is very much a consequence of the approach taken to prove the theorem.   Concerning  this the reader is especially directed  towards Remark \ref{gap}  at the end of \S\ref{enca} in the proof.}

\end{rem}

\medskip

\begin{rem} \label{rem3}  It is easily verified that if $f$ and $g$ are dimension functions satisfying \eqref{sv1} and $\cH^f(A) >0  $ then $\cH^g(A) = \infty $.  Thus,   the main substance of part (ii) of  Theorem \ref{mainthm} is the statement that $\cH^g(\proj_{\theta}A)  \ge \cH^g(A) $ for  almost all $\theta\in[0,\pi)$. For further relations between measures with respect to different gauge functions, see \cite[Section 4]{Rog}.
 \end{rem}

\medskip

\begin{rem} \label{ohyes} Regarding the dimension function $g$,  the condition  that $c < 2$  on the doubling constant  is necessary.   To see this, we derive the dimension aspect of Marstrand's Theorem from our result.  With this in mind,  assume without loss of generality that $\dim A > 0  $ and let $s_1,s_2$ be arbitrary  real numbers satisfying  $ 0< s_1 <  s_2 <  \dim A $.  Now let $g$ and $f$ be dimension functions given by $ g(r) := r^{s_1} $  and $ f(r):=r^{s_2}$.    It follows from the definition of Hausdorff dimension that $\cH^{s_1}(A) = \cH^{s_2}(A) = \infty$.   Also it is easily checked that condition \eqref{(1)} is satisfied  and thus, modulo the condition on the doubling  constant,  part (ii) of Theorem~\ref{mainthm} implies that $\cH^{s_1}({\proj}_{\theta}A) =\infty$  for almost all $\theta \in[0,\pi)$.   In turn, it follows (from the definition of Hausdorff dimension) that
\begin{equation}\label{triv2}
\dim {\proj}_{\theta}A  \ge s_1   \end{equation}
for almost all $\theta \in[0,\pi)$. The application of Theorem \ref{mainthm}
is legitimate  as long as the doubling constant $c= 2^{s_1}$ associated with $g$ %(see Remark \ref{rem1})
satisfies  $s_1 < 1$.   Now with reference to \eqref{triv2} this restriction on $s_1$ makes perfect sense since $\dim {\proj}_{\theta}A  \le 1 $ regardless of the size of  $A$.  By continuity, we can replace $s_1$ in  \eqref{triv2} by  $\dim A $. The complementary  upper bound can easily be deduced via part (i) of Theorem~\ref{mainthm} but  inequality \eqref{triv} gives it directly.
 \end{rem}

 \medskip

\begin{rem} Even  if $\cH^f(A) = \infty  $, the conclusion of  part (ii) of Theorem \ref{mainthm} is not in general valid for the dimension function $f$. Indeed, if  $f$ is given by $f(r):=r$ so that  $\cH^f$ is simply $1$-dimension Lebesgue measure,  it is known \cite[Section 6.4]{F} that there are  sets $A$ for which  $\cH^f(A) > 0 $ but $\cH^f(\proj_{\theta}A) = 0 $ for  almost all $\theta\in[0,\pi)$.
 \end{rem}

\medskip

As alluded to in Remark \ref{exceptremark}, in \S\ref{exceptangles}  we will  investigate  the size of the set of exceptional angles $\theta$ for which the conclusion of  part (ii) of Theorem \ref{mainthm} fails.   In short, by replacing the integral convergence  condition \eqref{(1)} by a suitable rate of   convergence condition we are able to conclude that the exceptional set of $\theta\in[0,\pi) $ for which  $\cH^g(\proj_{\theta}A)<\infty$  is of  $\cH^f $-measure 0, see Theorem \ref{p3}   for the precise statement.

\subsubsection{The logarithmic  dimension result   \label{mainlogsec}}

In terms of dimension theory,  when we are confronted with
sets of Hausdorff dimension 0 it is natural to change the usual
`$r^s$-scale' in the definition of Hausdorff dimension to a
logarithmic scale. For $s
> 0$, let $f_s$ be the dimension function given by
$f_s(r):=
(-\log^{*} \! r)^{-s}$, where
 $$
\log^*\! r:=\left\{\begin{array}{ll}  \log r &\mbox{for
} r \in (0,\frac12)  \\[1ex]
\log \frac12 &\mbox{for }r\ge \frac12 \, .
\end{array}\right.
$$
[The form of $\log^{*} \! r$ for $r\geq \frac12$ is to ensure that it is defined for $r\geq 1$; as remarked earlier, the particular form for $r\geq \frac12$ is of no consequence.]
 The {\em logarithmic Hausdorff dimension}  $\dim_{\rm \log} A$     of a set
$A$  is given by
\begin{equation}\label{logdim}
\dim_{\rm \log}  A  :=  \inf
\left\{ s : \mathcal{ H}^{f_s} (A) =0 \right\} = \sup \left\{ s :
\mathcal{ H}^{f_s} (A) = \infty \right\} \, .
\end{equation}
 It is easily verified
that if $ \dim A > 0 $ then $ \dim_{\rm \log}  A =
\infty$, precisely as one would expect.

Armed with Theorem \ref{mainthm} it is straightforward to prove  Theorem \ref{logResult}.

\medskip

\noindent{\em Proof of Theorem \ref{logResult}.  }
Part (i) is immediate from Theorem \ref{mainthm}(i)  and \eqref{logdim}.

For part (ii), without loss of generality, assume  that $\dim_{\rm \log}A  > 0  $ and let $s_1,s_2$ be  real numbers satisfying  $ 0< s_1  < s_2 < \dim_{\rm \log}A  $.  Let $g$ and $f$ be dimension functions given by $ g(r) := ( -\log^{*} \! r)^{-s_1} $  and $ f(r):=( -\log^{*} \! r)^{-s_2} $.       It  follows from \eqref{logdim}  that $\cH^{g}(A) = \cH^{f}(A) = \infty$.   It is easily verified that  condition \eqref{(1)} is satisfied and that
%for each  $s_1 >0$, there exists $r_0=r_0(s_1) > 0$ such that
$g$ is doubling with  constant $c < 2$ for $r\in (0,r_0)$ for some $r_0$. Thus, Theorem \ref{mainthm}(ii)  implies that $\cH^{g}({\proj}_{\theta}A) =\infty$  for almost all $\theta \in[0,\pi)$.   In turn, it follows from \eqref{logdim} that for almost all $\theta \in[0,\pi)$, $ \dim_{\rm \log} {\proj}_{\theta}A \ge s_1   $   and  thus  $ \dim_{\rm \log} {\proj}_{\theta}A \ge \dim_{\rm \log}A$.
\hfill $\Box$

\medskip

\subsubsection{Explicitly exposing the  gap of uncertainty \label{inad}}

With reference to our motivating example,  for all $\tau > 0$ let $\psi_\tau$ be the `approximating' function given by
$$
\psi_\tau(q) :=   q^{-\tau} ( \log q)^{-\tau}    \qquad  (q\geq 1).
$$
It follows, via  the Khintchine-Jarn\'{\i}k Theorem and the definition of  Hausdorff dimension,  that for all $ \tau \ge  (1+k)/k$
$$
\delta= \delta(\tau) := \dim W_k(\psi_\tau)  =  \frac{k+1}{\tau}  \, . $$  In fact, the Khintchine-Jarn\'{\i}k Theorem  implies  a much   finer conclusion.  Fix $ \tau >  (1+k)/k$  and  consider the family of dimension functions $(f_{\delta,s})_{s>0}$  given by
$$ f_{\delta,s}(r):= r^{\delta}  (-  \textstyle{\frac{1}{\tau}} \log^{*} \! r)^{s}  \, . $$
It is easily verified that
$$\sum_{q=2}^{\infty}  \ \; q^k \,
f_{\delta,s}\left(\p_\tau(q)\right)   \, \asymp \,   \sum_{q=2}^{\infty}  \ \; \frac{1}{q (\log q)^{1+k-s}}, $$
in the sense that the series either both converge or diverge,
and so the Khintchine-Jarn\'{\i}k Theorem  implies that
$$\mathcal{ H}^{f_{\delta,s}}\left(W_k(\p_\tau)\right)=\left\{\begin{array}{cl} 0
& {\rm \ if} \;\;\; s < k  \; ,\\[2ex]
\infty  & {\rm \ if} \;\;\; s \ge k         \; .
\end{array}\right.$$
Loosely speaking,  the set  $ W_k(\p_\tau) $  has ``$ \delta(\tau)$-logarithmic dimension'' equal to $k$.

\vspace*{3ex}

Now  let $k=2$ and with reference to Theorem \ref{mainthm}, put $f= f_{\delta,2} $ and $g=f_{\delta,s}$.   Suppose  that $ \tau >  3$ so that  $\delta(\tau) < 1$.  This ensures that $g$ is doubling with constant $c < 2$.  Theorem \ref{mainthm}  then implies that for  almost all $\theta\in[0,\pi)$
$$\mathcal{ H}^{f_{\delta,s}}\left( \proj_{\theta} \, W_2(\p_\tau)\right)=\left\{\begin{array}{cl} 0
& {\rm \ if} \;\;\; s < 2  \; ,\\[2ex]
\infty  & {\rm \ if} \;\;\; s > 3         \; .
\end{array}\right.$$
Of course,  part (i) of Theorem \ref{mainthm} implies that the  zero measure statement associated with  $s<2$ is true for all $\theta\in[0,\pi)$.  Regarding the application of part (ii),  we need that $s > 3$ in order  to satisfy the  integral convergence  condition \eqref{(1)}.   Thus the latter gives rise to a gap of uncertainty; namely   $ s \in (2,3)$ in the specific example under consideration.   We suspect that the infinity  measure statement for  $s>3$  is actually true for $s> 2$.

 \vspace*{1ex}

 {\em Problem: }    Show that  $\mathcal{ H}^{f_{\delta,s}}\left( \proj_{\theta} \, W_2(\p_\tau)\right)=   \infty $   if  $ s >  2  $.

\vspace*{1ex}

\noindent The fact of the matter is that it is highly unlikely that any set  $  W_2(\p)  $ of simultaneously
$\p$--approximable points  will have the necessary `dense rotational' structure that underpins the  construction of the sets associated with Theorem A$'$.

%\comsanju{OLD VERSION \noindent This would indeed be situation if we could replace condition \eqref{(1)} by condition \eqref{sv1} in the statement of Theorem \ref{mainthm}  -- see Remark \ref{rem2}.}
%

\section{Proof of main result}

Our proof of Theorem \ref{mainthm}  will  follow  Kaufman's  potential theoretic proof  \cite{K} of Mastrand's Theorem.  We adapt the proof that he gave for the specific  functions  $f(r)= r^s$ $(s > 0) $ to general dimension functions.

\subsection{Preliminaries:   doubling revisited and Frostman}  \label{prel}

We start by stating an equivalent form of the doubling condition  \eqref{double}.

\begin{lemma}  \label{double=} Let  $f$ be a dimension function.  Then $f$ is doubling if and only if there  exist constants $s >0 $, $ \kappa >0 $ and $r_1 > 0$ such that
\begin{equation}  \label{double2}
f(r \lambda)   \, \geq   \, \kappa  \; \lambda^s f(r)    \qquad   \forall \  0 <  \lambda < 1  \ \ {\rm and } \ \ 0< r < r_1 \, .
\end{equation}
Moreover,  if $f$ has a doubling constant $c>1$ then  \eqref{double2}  holds with $\kappa = c^{-1} $ and $s=\log_2c$.
\end{lemma}

This equivalence is essentially `folklore' and the exponent  $s$ appearing in \eqref{double2} is referred to as the {\em doubling exponent} of $f$.    Nevertheless,  for the sake of completeness we include the short proof.

\begin{proof}

Suppose $f$ has a doubling  constant $c>1$.  For each positive integer $n$,  applying   \eqref{double}  $n$ times gives that
$$ f\left(r\right) \  \le \  c^n \,  f( 2^{-n} r)  \quad   \forall  \ r< 2^{n} r_0 \, . $$
 Put $s:=\log_{2}c>0$.   For each  $0<\lambda <1$, let $m \ge 0$ be the unique integer such that  $2^m\leq \lambda^{-1} < 2^{m+1}$. Then
$$
f(\lambda r) \ge f\left(2^{-(m+1)}r\right)   \ \ge \  c^{-(m+1)} \,  f(r) \ \geq c^{-1} \lambda^{s} f(r)  \, .
$$
For the converse implication simply put $\lambda =\frac12 $ in \eqref{double2}.
\end{proof}

\vspace*{3ex}

The following statement is a generalisation   of Frostman's  fundamental lemma  to  arbitrary dimension functions $f$ .  Throughout,  given a Borel set $ A\subset \R^2$ we  denote by $\cM_1(A)$ the set  of Radon probability  measures $ \mu$ with compact  support in $A$.

\begin{theorem}[Frostman's Lemma]\label{frostman}
Let $A\subseteq \mathbb{R}^2$ be a Borel set and $f$ be a dimension function. Then $\cH^{f}(A)>0$ if and only if there exist a  measure $\mu \in \cM_1(A)$  and a constant $c_1 > 0  $ such that
$$ \mu(B(x,r)) \leq c_1  f(r)    \qquad  \forall \  x\in\R^2    \ \ {\rm and } \ \  r>0 \, . $$
\end{theorem}
\begin{proof}
Two  very different proofs for the case where  $f(r)= r^s$ $(s >0)$ are given in \cite[Theorem 8.8]{Mat}, where is explicitly pointed out that both proofs are valid for general dimension functions. Alternatively, for the harder of the implications, namely that suitable measures exist, the result in Rogers \cite[Theorem 57]{Rog}, that for general dimension functions there exists a compact subset  $A'$ of $A$ with   $0<\cH^{f}(A')<\infty$, followed by a density argument akin to \cite[Proposition 4.11]{F}, also gives the conclusion.
\end{proof}

\subsection{ Energies and capacities} \label{enca}

 We first generalise the standard notions of $s-$energy and $s-$capacity of a measure, see for example \cite[Section \!4.3]{F} and \cite[Chapter 8]{Mat}.  As usual, let $f$ be a dimension function.

\begin{definition} The {\em $f$-energy} of $ \mu \in \cM_1(A)$ is defined as
$$ I_{f}(\mu) := \iint\!\frac{\mathrm{d}\mu(x)\mathrm{d }\mu(y)}{f(|x-y|)} \cdot$$
\end{definition}

\noindent Alternatively, we could have defined the $f$-energy via the {\em $f$-potential} at a point $x \in \R^2$, that is
$$
 \phi_f(x) := \int\frac{\mathrm{d}\mu(y)}{f(|x-y|)} \quad {\  \rm and \ so \quad } I_{f}(\mu) =  \int  \phi_f(x)\  \mathrm{d}\mu(x) .
$$

\begin{definition}\label{defcap}
 The {\em $f$-capacity} of a Borel set $ A\subset \R^2$ is defined as
$$ C_f(A) := \sup\left\{ \frac{1}{I_f(\mu)} : \mu \in\cM_1(A)  \right\}$$
with the interpretation that $C_f(\emptyset) = 0$. \end{definition}

\noindent Naturally, when $f$ is given by $f(r)= r^s$ $(s >0)$ we recover  the familiar notions of $s$-energy and $s$-capacity.

We now establish the connection between the Hausdorff measure $\cH^f(A)$ and the capacity $C_{f}(A)$ of a set $A$ with respect to a general dimension function $f$. These results stated below have a long history: apart from notational differences they appear as Theorems~1 and 2   in~\cite{T}, though  versions for the dimension functions of the form $f(r)= r^s$   date back to the 1930s. The paper  \cite{T} discusses the historical development to increasingly general dimension functions and includes further references. Proofs for  dimension functions  $f(r)= r^s$ may be found in several more recent accounts of fractal geometry, for example \cite{F1,Mat}.
Even for general dimension functions the proofs are relatively short, so for the sake of clarity, consistency of notation  and completeness we  include the proofs.

%The above proofs of Proposition \ref{prop6}  and  Proposition \ref{p1}   follow those of~\cite[Theorem 8.7(1)]{Mat} and \cite[Theorem 8.13]{Mat} respectively.  The proofs there are for dimension functions $f$ given by $f(r)= r^s$ $(s >0)$.  Our  proofs are suitably adapted to the situation of general dimension functions.

\begin{proposition} \label{prop6}
Let $A\subseteq \mathbb{R}^2$ be a Borel set and $f$ be a dimension function.  If $\mathcal{H}^f(A)<\infty$ then $C_f(A)=0$.
\end{proposition}

\begin{proof}
Assume $C_f(A)>0$. By definition, the set $A$ supports a  Radon probability measure $\mu$ such that $I_f(\mu)<\infty$.  Thus
$$ \int\frac{\mathrm{d}\mu(y)}{f(|x-y|)} < \infty \text{ \ \ for $\mu$-almost all $x\in A.$} $$
For such $x\in A$,
$$ \lim_{r\to 0} \mathop{\int}_{B(x,r)}\frac{\mathrm{d}\mu(y)}{f(|x-y|)} = 0. $$
By Egorov's theorem,  for all $ \ep > 0$ there exist $\delta>0$ and a Borel set $K\subseteq A$ such that $\mu(K)>\frac{1}{2}$ and
\begin{eqnarray*}
\mu(B(x,r))  &\leq & f(r) \mathop{\int}_{B(x,r)}\frac{\mathrm{d}\mu(y)}{f(|x-y|)} \\[2ex]
&\leq& \varepsilon f(r)  \text{ \ \  for all   $x\in K $   and  \ $ 0<r\leq \delta$}.
\end{eqnarray*}

Now let $\big(B(x_i,r_i)\big)_{i=1}^{\infty}$ be a cover of $K$ by balls with $x_i \in K$ and  $r_i \leq \delta$ such that
$$ \sum_{i=1}^{\infty}f(r_i) \ <\  \cH_C^{f}(K) +1\  $$
Then
$$
\frac{1}{2}<\ \mu(K) \leq \sum_{i=1}^{\infty} \mu(B(x_i,r_i))
\leq  \ve \,   \sum_{i=1}^{\infty}f(r_i)
\leq \ve \,   \left( \cH_C^{f}(A)+1\right),
$$
where $\cH_C^{f}$ is centred Hausdorff measure.
Since $\ve>0$ can be made arbitrarily small, we conclude that $\cH^{f}(A)= \cH_C^{f}(A)=\infty$, using \eqref{mesequiv}. This contradicts our hypothesis that $\cH^{f}(A)$ is finite.
\end{proof}

%Now let $(B(x_n,r_n))_{n=1}^{\infty}$ be a $\delta$-cover of $K$ such that  $K\cap B(x_n,r_n)\neq \emptyset \ (n=1,2, \ldots$) and
%$$ \sum_{n=1}^{\infty}f(r_n) < \cH^{f}(A) +1. $$
%Choose $y_n\in K\cap B_n $  $ (n=1,2,\ldots $) and note that for each $n$ we trivially have that  $  B(x_n,r_n) \subset B(y_n,2r_n)$. It then follows,  on using the fact that $f$ is doubling (with constant $c$),   that
%\begin{eqnarray*}
%\frac{1}{2}<\mu(K) &\leq& \sum_{n=1}^{\infty} \mu(B(y_n,2r_n)) \\
%&\leq & \ve \, c \,  \sum_{n=1}^{\infty}f(r_n) \\[2ex]
%&\leq& \ve \, c \,   \left( \cH^{f}(A)+1\right).
%\end{eqnarray*}
%Since $\ve>0$ can be made arbitrarily small, we conclude that $\cH^{f}(A)=\infty$. This contradicts our hypothesis that $\cH^{f}(A)$ is finite.

\vspace*{0.5cm}

\begin{proposition} \label{p1}
Let $A\subseteq \mathbb{R}^2$ be a Borel set and let $f$ and $g$ be dimension functions satisfying the integral convergence condition \eqref{(1)}.
If $\mathcal{H}^{f}(A)>0$ then $C_{g}(A)>0$.
\end{proposition}
\begin{proof}
%In view of Theorem~\ref{hoy}, there exists a compact set $K \subseteq A$  such that  $0<\cH^{f}(K)<\infty$. In turn,  Theorem \ref{frostman} implies that $K$ supports a Radon probability measure $\mu$  such that
By Frostman's lemma, Theorem \ref{frostman},  the Borel set $A$  supports a Radon probability measure $\mu$  such that
\begin{equation} \label{goody} \mu\left( B(x,r)\right)   \, \leq  \, c_1  \,  f(r)   \text{ \ \ $ \forall   \ \ x\in \R^2 $  \   and  \ $ r > 0 $} \end{equation}
for  some constant $c_1>0$. Fix $x\in \R^2$ and let
$$ m(r) := \mu\left( B(x,r)\right)  \, . $$
Using   \eqref{goody} and   that $\mu(\R^2) = \mu( K)=1$ and integrating by parts,
\begin{eqnarray*}
 \int \frac{\mathrm{d}\mu(y)}{g(|x-y|)} \  &=& \! \mathop{\int}_{|x-y|\leq 1}  \frac{\mathrm{d}\mu(y)}{g(|x-y|)} \ +\mathop{\int}_{|x-y|>1}  \frac{\mathrm{d}\mu(y)}{g(|x-y|)} \  \\[2ex]
& \leq & \int_{0}^{1}\frac{1}{g(r)}  \ \mathrm{d}m(r) \ + \  \frac{\mu(\mathbb{R}^2)}{g(1)} \\[2ex]
&=& \left[ \frac{m(r)}{g(r)}\right]_{0}^{1} \ -  \ \int_{0}^{1}m(r)  \ \mathrm{d}\left( \frac{1}{g(r)}\right)  \ + \   \frac{\mu(\mathbb{R}^2)}{g(1)}  \\[2ex]
& \leq & \frac{m(1)}{g(1)}  \ - \ \lim_{r\to 0^+}\frac{m(r)}{g(r)} \ - \  \int_{0}^{1}\!f(r)\ \mathrm{d }\left(\frac{1}{g(r)}\right) \ + \  \frac{\mu(\mathbb{R}^2)}{g(1)} \\[2ex]
&<& \infty,
\end{eqnarray*}
noting that $m(r)/g(r) \leq c_1 f(r)/g(r) \to 0$ by \eqref{sv1}.
This bound is uniform for all $x\in \R$  and so
$$  I_{g}(\mu) = \iint  \frac{\mathrm{d}\mu(x)\mathrm{d}\mu(y)}{g(|x-y|)}  \  <  \  \infty  \,   $$
giving $C_{g}(A)>0$ by Definition \ref{defcap}.
\end{proof}

%\hidden{
%We finish the section with a lemma connecting the $f$-capacity of a set with that of its projection on the line $L_{\theta}.$
%
%
%\begin{lemma} \label{thelem}
%Let $A\subset \mathbb{R}^2$ and $0\leq \theta< \pi$. For any dimension function $f$ we have $$ C_{f}(\proj _{\theta}A) \leq C_f(A). $$
%\end{lemma}
%\begin{proof}
%Let $\nu$ be a measure supported on $\proj_{\theta}A$. Define a measure $\mu$ on $A$ by setting $\mu(B)=\nu(\proj_{\theta}(B)), \ B\subset A$. Then
%\begin{eqnarray*}
%I_f(\nu) &=& \iint\! f(|t-s|)^{-1}\mathrm{d \!}\nu(t) \mathrm{d \!}\nu(s) \\
%&=& \iint f(|\proj_{\theta}x-\proj_{\theta}y|)^{-1} \mathrm{d \!}\mu(x)\mathrm{d \!}\mu(y) \\
%&\geq & \iint f(|x- y|)^{-1}\mathrm{d \!}\mu(x)\mathrm{d \!}\mu(y) \\
%&=& I_f(\mu) \\
%&\geq& C_{f}(A)^{-1}.
%\end{eqnarray*}
%This is true for all probability measures supported on $A$, hence
%$$
%C_{f}(\proj_{\theta}A)^{-1} = \inf \left\{ I_f(\nu): \nu\in\cM_1(\proj_{\theta}(A)) \right\}  \geq
%C_{f}(A)^{-1}
%$$
%and the requested follows.
%\end{proof}
%}

\vspace*{3ex}

\begin{rem} \label{gap}
Fix $ 0< \delta < 1 $  and  consider the family of dimension functions $(f_{\delta,s})_{s>0}$  given by
$$ f_{\delta,s}(r):= r^{\delta}  (  - \log^* \! r)^{s}  \, , $$
to within constants the same as those considered in \S\ref{inad}.  Let $A\subseteq \mathbb{R}^2$ be a Borel set  and $\alpha > 0$.   Then, by Propositions  \ref{prop6} and \ref{p1},
\begin{itemize}
\item[(i)] if $s \le \alpha $   and $\mathcal{H}^{f_{\delta,\alpha}}(A) < \infty$ then $C_{f_{\delta,s}}(A)=0$,
    ~\vspace*{0ex}
    \item[(ii)] if $s > \alpha +1  $   and $\mathcal{H}^{f_{\delta,\alpha}}(A) > 0 $ then $C_{f_{\delta,s}}(A)>0$.
\end{itemize}
%The fact that we require $s > \alpha +1  $ in (ii) is to ensure that the
% integral convergence condition \eqref{(1)} is satisfied.
 The upshot is that if $\alpha < s \le \alpha +1  $, condition \eqref{(1)} is not satisfied  and the propositions provide no information.  The main aim of the paper \cite{T} is to expose  this  gap of uncertainty.    So for example, by \cite[Theorem 3]{T},  if $f$ and $g$ are dimension functions not satisfying  condition \eqref{(1)}, then there exist Borel sets $A$  with $0<\mathcal{H}^f(A)<\infty$ but $C_g(A)=0$.
\end{rem}

\subsection{Proof of Theorem \ref{mainthm}}

\noindent  \emph{(i)}   As pointed out in Remark \ref{part1}, this is a trivial consequence of  the definition of the Hausdorff measures  that projection is a Lipschitz mapping.

\vspace{2ex}

\noindent \emph{(ii)} From  Remark \ref{rem3},  $\cH^{g}(A)= \infty$.  Thus it suffices to show that $\cH^g(\proj_{\theta}A)  = \infty  $ for  almost all $\theta\in[0,\pi)$.

Since $\cH^{f}(A)>0$,   it follows via  Proposition \ref{p1}  and the  definition  of capacity, that $A$ supports a Radon probability measure $\mu$  such that  $I_{g}(\mu)<\infty.$  For each  $\theta\in[0,\pi)$, projecting $\mu$ onto the line $L_{\theta}$ gives a  measure $\mu_{\theta}$ supported on $\proj_{\theta}A$ defined by the requirement that $\mu_{\theta}(K)=\mu(\proj_{\theta}^{-1}(K)) $ for each Borel set $K \subset L_{\theta}$.  For each $x\in\mathbb{R}^2$, let   $\phi(x)$ denote the angle that $ x$  (viewed as  a vector)  forms with the horizontal axis. Then,  by Lemma  \ref{double=}  and using the fact that $g$ is doubling with constant $c<2$, it follows that

 \begin{eqnarray*}
\int_{0}^{\pi}I_{g}(\mu_{\theta})\mathrm{d}\theta &=&\int_{0}^{\pi} \!\iint \frac{\mathrm{d}\mu_{\theta}(x)\mathrm{d}\mu_{\theta}(y)}{g(|x-y|)}\mathrm{d}\theta \\[2ex]
&=& \int_{0}^{\pi}\!\iint\frac{\mathrm{d}\mu(x)\mathrm{d}\mu(y)}{g(|\proj_{\theta}x-\proj_{\theta}y|)}\mathrm{d}\theta \\[2ex]
&\leq & \iint\left( \int_{0}^{\pi}\frac{c}{g(|x-y|)|\cos(\phi(x-y)-\theta)|^s} \mathrm{d}\theta\right) \mathrm{d}\mu(x)\mathrm{d}\mu(y)  \\[2ex]
&\leq& c_1\iint\frac{\mathrm{d}\mu(x)\mathrm{d}\mu(y)}{g(|x-y|)} \text{ \ \ \  (because $s=\log_2 c <1$) } \\[2ex]
&=& c_1 I_g(\mu)  \ < \ \infty \, .
\end{eqnarray*}
This implies that $I_{g}(\mu_{\theta})<\infty$ for almost all $\theta \in [0,\pi)$.   From the definition of capacity, $C_{g}(\proj_{\theta}A)>0$ for such $\theta$, so by Proposition~\ref{prop6}, $\cH^{g}(\proj_{\theta}A)=\infty$ for almost all $\theta \in [0,\pi)$.
\hfill $\Box$

\section{Exceptional  projections} \label{exceptangles}

 Marstrand's Theorem trivially implies that the set of exceptional angles
 $$
 E(A) := \left\{\theta \in [0,\pi)  \, :  \, \dim {\proj}_{\theta}A  < \dim A \right\}  \, ,
 $$
 is a set of (one-dimensional) Lebesgue measure zero.   Kaufman also showed \cite{K} that
 \begin{equation}  \label{dimkaufman}
 \dim  \, E(A)  \,  \, \le \, \min \{ 1,  \dim A  \}  \,
 \end{equation}
 (see also Remark \ref{matthm4} below).
 Clearly,  when $\dim A < 1$,  this bound on the size of the set of exceptional angles  is significantly stronger than the measure zero statement of Marstrand's Theorem.  It is natural to attempt to extend  Theorem~\ref{mainthm} in a similar fashion.    With this in mind, let  $E_g(A) $ denote the  exceptional set of  $\theta\in[0,\pi)$ for which  the conclusion of part (ii) of  Theorem \ref{mainthm} fails; that is
 \begin{equation}  \label{ega}
  E_g(A) := \left\{ \theta\in [0,\pi): \cH^{g}({\proj}_{\theta}A)<\infty \right\} \, .
 \end{equation}
By replacing  the integral convergence  condition \eqref{(1)} by a rate of   convergence condition we are able to establish the following strengthening  of  Theorem \ref{mainthm}. It is easily verified that  condition \eqref{1+}  below  implies  condition \eqref{(1)}  of  Theorem \ref{mainthm}.

%
% Indeed,  it is trivial if $t_0 \ge 1$.  Simply put  $t=1$ in \eqref{1+} and  \eqref{(1)} follows. For  $ t_0  < 1 $, observe that for any $0<t_1<t_0$,
%\begin{eqnarray*}
%\infty  \ >   \  \int_{0}^{1}f(r)  \ \mathrm{d \!}\left( \frac{-1}{g(t_1r)} \right) &=& \int_{0}^{1} f(r)\frac{g'(t_1r)}{g^2(t_1r)}\mathrm{d}(t_1r) \\[2ex]
%&=& \int_{0}^{t_1}f\left(\frac{r}{t_1} \right)\mathrm{d \!}\left(\frac{-1}{g(r)} \right) \\[2ex]
%&\gg& \int_{0}^{t_1}f\left(r \right)\mathrm{d \!}\left(\frac{-1}{g(r)} \right)  \, .
%\end{eqnarray*}

\begin{theorem} \label{p3} Let $A\subseteq \mathbb{R}^2$ be a Borel set. Let $f$ be a dimension function such that  $\cH^f(A) > 0 $ and let $g$ be a dimension function that is  doubling. Suppose that   there exist  constants  $  t_0 $ and   $ c_2 >  0$ such  that
\begin{equation} \label{1+}
 -\int_{0}^{1}f(r) \ \mathrm{d \!}\left(\frac{1}{g(tr)}\right)  \,  <   \, c_2 \,  \frac{1}{g(t)}  \  \quad  \text{ for all }\  \    0 <  t < t_0 \, .
\end{equation}
Then,  $ \mathcal{H}^{f} ( E_{g}(A)  ) = 0$.
\end{theorem}

\begin{proof}
In view of Proposition \ref{prop6},
$$ E_{g}(A) \subset   E_{*} := \left\{ \theta\in [0,\pi): C_g({\proj}_{\theta}A)=0\right\} \, . $$
Thus, it suffices to show that $\cH^{f}(E_{*})=0$. Suppose this is not the case.   Then $\cH^{f}(E_{*})>0$ and by Theorem \ref{frostman} the set  $E_{*}$ supports a probability measure $\nu\in\mathcal{M}_1(E_{*}) $  such that
\begin{equation*}
\label{goody3} \nu\left( B(x,r)\right)   \, \leq  \, c_1  \,  f(r)   \text{ \ for all   $\  x\in \R^2 $ \ and  \ $ r > 0 $} \end{equation*}
where $c_1>0$ is an  absolute  constant.  On the other hand,  since $\cH^{f}(A)>0$ and condition \eqref{1+}  implies  condition \eqref{(1)}, it follows  via  Proposition \ref{p1}  and the definition of capacity, that $A$ supports a  probability measure $\mu \in\mathcal{M}_1(A) $  such that
\begin{equation} \label{yum}
I_{g}(\mu)<\infty.
\end{equation}

For each  $\theta\in[0,\pi)$, let  $\mu_{\theta}$ be the projection of   $\mu$ onto the line $L_{\theta}$  supported on $\proj_{\theta}A$ such that $\mu_{\theta}(K)=\mu(\proj_{\theta}^{-1}(K)) $ for each Borel set $K \subset L_{\theta}$  -- as in the proof of Theorem~\ref{mainthm}.
 Let us assume for the moment that
\begin{equation}
\label{desire}
\int_{E_{*}}\!I_g(\mu_{\theta})\mathrm{d}\nu(\theta) <\infty  \, .
\end{equation}
This implies that $I_{g}(\mu_{\theta})<\infty$ for $\nu$-almost all $\theta \in E_{*}$. By the definition of capacity,  $C_{g}(\proj_{\theta}A)>0$ for such $\theta$, contradicting that $C_g({\proj}_{\theta}A)=0$ if   $\theta \in E_{*}$.  This completes the proof of the theorem modulo establishing \eqref{desire}.

To establish \eqref{desire},  we first observe that for all $x\in\mathbb{R}^2\setminus\{0\}$ and $d > 0 $,  the set
$$\left\{ \theta\in[0,\pi): |{\proj}_{\theta}x |\leq d\right\} $$
is a union of at most two intervals each of diameter at most
$
\pi \,  d/| x |
$.
 The upshot is that
$$
\nu \left(\left\{ \theta\in[0,\pi): |{\proj}_{\theta}x |\leq d\right\}\right)  \  \leq \    2 \, c_1 f \Big( \pi \frac{ d}{| x |}   \Big) \,  .   $$
This, together with the fact that $g$ is doubling, implies that
\begin{eqnarray*}
\int_{E_{*}} \ \frac{1}{g(|{\proj}_{\theta}x|)} \ \mathrm{d}\nu(\theta) &=& \int_{0}^{\infty} \nu\left(\Big\{\theta: \frac{1}{g(|{\proj}_{\theta}x|)}\geq r  \Big\} \right)\mathrm{d}r \\[2ex]
&=& \int_{0}^{1/g(|x|)} \nu\left(\Big\{\theta: \frac{1}{g(|{\proj}_{\theta}x|)}\geq r  \Big\}\right)\mathrm{d}r \\[2ex]
& & \ \quad + \quad   \ \int_{1/g(|x|)}^{\infty} \nu\left(\Big\{\theta: \frac{1}{g(|{\proj}_{\theta}x|)}\geq r  \Big\} \right)\mathrm{d}r \\[3ex]
&\leq  & \frac{1}{g(|x|)}  \ +  \  \int_{1/g(|x|)}^{\infty} 2 \, c_1 f\left(\frac{\pi}{|x|} \, g^{-1}\left(\textstyle{\frac{1}{r}} \right) \right) \mathrm{d}r \\[2ex]
&\leq & \frac{1}{g(|x|)}  \ - \ 2c_1 \ \int_{0}^{\pi}\! f(u) \ \mathrm{d\!}\left(\frac{1}{g(|x|u/\pi)} \right) \\[2ex]
&=& \frac{1}{g(|x|)}  \ + \ 2c_1 \ \int_{0}^{1}\! f(u) \ \mathrm{d\!}\left(\frac{-1}{g(|x|u/\pi)} \right) \\[2ex]
& & \ \quad + \quad   2c_1 \ \int_{1}^{\pi}\! f(u) \ \mathrm{d\!}\left(\frac{-1}{g(|x|u/\pi)} \right) \\[2ex]
& \stackrel{\eqref{1+}}{\leq} & \frac{1}{g(|x|)}  \ + \ 2c_1 c_2  \ \frac{1}{g(|x|/\pi)}  \\[2ex]
& & \ \quad + \quad   2c_1 \, f(\pi)\ \left(\frac{1}{g(|x|/\pi)} -\frac{1}{g(|x|)} \right) \\[2ex]
&\leq & c_3  \  \frac{1}{g(|x|)}  \,
\end{eqnarray*}

\noindent for some $c_3$ and for all $x\neq 0$ with $|x|<t_0$ .  Hence, using  Fubini's theorem,
\begin{eqnarray*}
\int_{E_{*}}I_g(\mu_{\theta}) \ \mathrm{d}\nu(\theta) &=&
\int_{E_{*}}\iint  \ \frac{\mathrm{d}\mu_{\theta}(x) \, \mathrm{d}\mu_{\theta}(y)}{g(|x-y|)} \ \mathrm{d}\nu(\theta) \\[2ex]
&=& \int_{E_{*}}\iint \  \frac{\mathrm{d}\mu(x) \, \mathrm{d}\mu(y)}{g(|\proj_{\theta}x-\proj_{\theta}y|)} \  \mathrm{d}\nu(\theta) \\[2ex]
&=& \iint \int_{E_{*}} \  \frac{\mathrm{d}\nu(\theta)}{g(|\proj_{\theta}(x-y)|)} \ \mathrm{d}\mu(x) \, \mathrm{d}\mu(y) \\[2ex]
& \leq & c_3 \iint \ \frac{\mathrm{d}\mu(x) \, \mathrm{d}\mu(y)}{g(|x-y|)}   \ < \ \infty \,
\end{eqnarray*}
by \eqref{yum}. This establishes   \eqref{desire} and  completes the proof.
%where in the last step we assumed $g$ is a one-to-one function. After a change of variables, the integral in the last line is
% $$ - \int_{0}^{u_1}f(u)\mathrm{d}\left(\frac{1}{g(t u)} \right)   \qquad ( \, t > 0 \, )  $$
%with $u_1=\frac{1}{t}g^{-1}\left(\frac{g(t)}{K} \right)>0,$ so in order to get the required result it remains to show
%
% $$ - \int_{0}^{u_1}f(u)\mathrm{d}\left(\frac{1}{g(tu)} \right) \ll \frac{1}{g(t)} \ \ \text{ as } t \to 0. $$
%An alternative change of variables, namely $u=g^{-1}\left( \frac{1}{r}\right)$, shows the integral is equal to
%$$ -\int_{0}^{t}f\left(\frac{1}{t}u\right)\mathrm{d\!}\left( \frac{1}{g(u)} \right)$$
%so it would also be enough to show
%$$ -\int_{0}^{t}f\left(\frac{1}{t}u\right)\mathrm{d\!}\left( \frac{1}{g(u)} \right) \ll \frac{1}{g(t)}, \  t \to 0. $$
\end{proof}

\vspace*{3ex}

\begin{rem} \label{matthm4} The  above proof of Theorem \ref{p3}  is based on the proof of the dimension inequality  \eqref{dimkaufman} presented in  \cite[Theorem~5.1]{Mat2}.     Note that it is easy to deduce \eqref{dimkaufman} from Theorem \ref{p3}. Indeed, to see that this is the case, without loss of generality assume that $ 0<  \dim A <1  $ and let $s_1,s_2$ be  real numbers satisfying  $ 0< s_1  < s_2 < \dim A $.  Let $$
 E(A,s_1):= \{ \theta\in [0,\pi): \dim \proj_{\theta}A< s_1 \}.  $$
 Let $g$ and $f$ be dimension functions given by $ g(r) := r^{s_1} $  and $ f(r):=r^{s_2}$.    It follows that $\cH^{s_2}(A) = \infty$ and that $\cH^{s_1}  \big( \proj_{\theta}A  \big) = 0$  for all $\theta \in E(A,s_1)$. Thus
$$
 E(A,s_1)  \, \subset \, E_{g} (A)  \, ,
$$
with $E_{g} (A)$ as in \eqref{ega}. Clearly,  the function $g$ is doubling and it is easily checked that  $f$ and $g$ satisfy  condition  \eqref{1+}.  Theorem \ref{p3} implies that  $ \cH^{s_2}(E_g(A))=0  $ and so
$ \dim \big( E (A,s_1)  \big)  \le s_2 \,  $, and \eqref{dimkaufman} follows on taking $s_1, s_2$ arbitrarily close to $\dim A$.
\end{rem}

\begin{rem} \label{matthm4a} The  above proof of Theorem \ref{p3}  is based on the proof of the special case \eqref{dimkaufman} presented in  \cite[Theorem~5.1]{Mat2}.
\end{rem}

Armed with Theorem \ref{p3} it is straightforward to prove  \eqref{exceptequation} which we formally state as a corollary.

\vspace*{1ex}

\begin{corollary}\label{logcor} Let $A\subseteq \mathbb{R}^{2}$ be a Borel set.  Then,  $\dim_{\log}E_{\log} (A) \leq \dim_{\rm \log}A $ where
$$
 E_{\log} (A):= \{ \theta\in [0,\pi): \dim_{\rm \log}\proj_{\theta}(A)< \dim_{\rm \log}A \}.  $$
\end{corollary}

\begin{proof}
Without loss of generality, assume that $ 0<  \dim_{\rm \log}A < \infty $ and let $s_1,s_2$ be  real numbers satisfying  $ 0< s_1  < s_2 < \dim_{\rm \log}A $.  Let $$
 E_{\log} (A,s_1):= \{ \theta\in [0,\pi): \dim_{\rm \log}\proj_{\theta}(A)< s_1 \}.  $$
As in the proof of Theorem \ref{logResult},  let $g$ and $f$ be the dimension functions  $ g(r) := ( -\log^{*} \! r)^{-s_1} $  and $ f(r):=( -\log^{*} \! r)^{-s_2} $.  Then $ \cH^{f}(A) = \infty$ and $\cH^g \big( \proj_{\theta}(A)  \big) = 0$  for all $\theta \in E_{\log} (A,s_1)$ so
$
 E_{\log} (A,s_1)  \, \subset \, E_{g} (A)
$.
Clearly $g$ is doubling. Assume for the moment that $f$ and $g$ satisfy  condition  \eqref{1+}.  Then Theorem \ref{p3} implies that  $ \cH^{f}(E_g(A))=0  $ so from the definition of logarithmic Hausdorff dimension \eqref{logdim},
$$
 \dim_{\rm \log} \big( E_{\log} (A,s_1)  \big)  \le s_2 \, .
$$
The conclusion now follows on taking $s_1, s_2$ arbitrarily close to $\dim_{\rm \log}A $.
\vspace{2mm}

\noindent
It remains to  verify \eqref{1+}.  For all sufficiently small $t > 0 $,
\begin{eqnarray*}
-\int_{0}^{1}f(r)  \ \mathrm{d \!}\left(\frac{1}{g(tr)}\right)
&=& \left[ \frac{f(r)}{g(tr)}\right]_{1}^{0}\ + \ \int_{0}^{1}\frac{\mathrm{d}f(r)}{g(tr)} \\[2ex]
&=& -\frac{f(1)}{g(t)}\hspace{1mm} +\hspace{1mm} \int_{0}^{t}\frac{\mathrm{d}f(r)}{g(tr)}\hspace{1mm} + \hspace{1mm} \int_{t}^{1}\frac{\mathrm{d}f(r)}{g(tr)} \cdot
\end{eqnarray*}
For the first integral, $r\leq t$   implies that
\begin{eqnarray*} \frac{1}{g(tr)} \leq 2^{s_1} (-\log r)^{s_1} = 2^{s_1}\frac{1}{g(r)}
\end{eqnarray*}
and hence it follows that \vspace{2mm}
\begin{eqnarray*}
\int_{0}^{t}\frac{\mathrm{d}f(r)}{g(tr)} &\leq& 2^{s_1}\int_{0}^{t}\frac{\mathrm{d}f(r)}{g(r)} \\[2ex]
&\leq & 2^{s_1}\int_{0}^{1}\frac{\mathrm{d}f(r)}{g(r)}\\[2ex]
&=& \frac{2^{s_1}s_2(s_2-s_1)}{(\log2)^{s_2-s_1}} \cdot
\end{eqnarray*}
For the second integral, $r\geq t$ implies that \vspace{1mm}
$$ \frac{1}{g(tr)} < 2^{s_1}\frac{1}{g(t)} $$
and hence it follows that  \vspace{1.5mm}
\begin{eqnarray*}
\int_{t}^{1}\frac{\mathrm{d}f(r)}{g(tr)} < 2^{s_1}\int_{0}^{1}\frac{\mathrm{d}f(r)}{g(t)} = 2^{s_1}\frac{f(1)}{g(t)} \cdot
\end{eqnarray*}
On combining these estimates, we obtain that  \vspace{3mm}
\begin{eqnarray*}
-\int_{0}^{1}f(r)  \ \mathrm{d \!}\left(\frac{1}{g(tr)}\right)& =&(2^{s_1}-1)f(1)\frac{1}{g(t)} + \frac{2^{s_1}s_2(s_2-s_1)}{(\log2)^{s_2-s_1}}\\
&\leq & c_2\frac{1}{g(t)}
\end{eqnarray*}
for some constant $c_2$, as desired.
%&=& \int_{0}^{1/2t} s_1 f(r)(-\log t-\log r)^{s_1-1}  \ \frac{\mathrm{d}r}{r} \\[2ex]
% &=& \int_{0}^{1/2t}s_1 f(r) \left(\frac{-\log t}{-\log r} + 1\right)^{s_1-1} \!\! (-\log r)^{s_1-1}   \ \frac{\mathrm{d}r}{r} \\[2ex]
% &\leq& \int_{0}^{1/2t}s_1 f(r) \left(\frac{-\log t}{\log 2} + 1\right)^{s_1-1} \!\!
% (-\log r)^{s_1-1}  \   \frac{\mathrm{d}r}{r} \\[2ex]
%%&<& \int_{0}^{1/2}s_1f(r)(c_1(-\log t)+1)^{s_1-1}(-\log r)^{s_1-1}\frac{\mathrm{d}r}{r} \\[2ex]
%%&  & \text{ \ (where } c_1=(\log 2)^{-1}>0) \\
%&\leq &  s_1(-\log t)^{s_1} \, \int_{0}^{1}f(r)  \ \mathrm{d \!} \left(\frac{-1}{g(r)} \right)\\[2ex]
%& \leq &  c_2 \frac{1}{g(t)}   \,
\end{proof}

\vspace{0ex}

\section{Final comments \label{fincomments}}

Apart from working in  higher dimensions, there are several other directions in which one could attempt to strengthen/generalize the main theorem.  We concentrate on just a few of  them.

\vspace{2ex}

\noindent{\em The gap of uncertainty. } Theorem A$'$  shows that  we can not in general replace condition \eqref{(1)}  by \eqref{sv1}  in Theorem ~\ref{mainthm}. Thus there is a gap of uncertainty associated with Theorem~\ref{mainthm}.    It would be highly desirable to know whether or not condition \eqref{(1)} is really necessary.  Namely, if $f$ and $g$ are dimension functions such that \eqref{(1)} is not satisfied, then does there exist a set $A \subseteq \R^2$ such that $\cH^f(A) > 0$ but $\cH^g(\proj_\theta A) = 0$ for almost all $0\leq \theta < \pi$?  Theorem A$'$  provides sufficient conditions on $f$ and $g$ for the existence of such a set $A$.

%We suspect   that  it is possible to replace the integral convergence condition  \eqref{(1)}   by \eqref{sv1} in the statement of Theorem \ref{mainthm}  -- see Remark \ref{rem2}.      The presence of   \eqref{(1)}   is almost certainly a  repercussion of the potential theoretic approach taken to prove the theorem.  Indeed,   as pointed out in  Remark~\ref{gap}, condition \eqref{(1)} and thus the resulting ``gap of uncertainty'' is unavoidable if the proof is based on Kaufman's approach of exploiting the  intimate relationship between Hausdorff measures and capacity.  Thus, in order  to replace  condition \eqref{(1)}   by \eqref{sv1}  in Theorem~\ref{mainthm}  we need to develop a  new approach and this is fundamentally  interesting in its own right.  Indeed, all  known proofs concerning the stronger statements for the size of the  sets of exceptional projections (such as   those appearing in \S\ref{exceptangles})  are  based on Kaufman's potential theoretic approach.

\noindent{\em Brownian paths. }
Brownian motion sample paths, see \cite[Chapter 16]{F} for a general introduction, illustrate the sort of situation that can arise for projections of sets in $\R^3$, and which perhaps may occur in $\R^2$, though there is no direct analogue. Let $B[0,1]\subseteq \R^3$ be  (random) Brownian motion path over the unit time interval. Then, almost surely, the Hausdorff dimension of  $B[0,1]$ is logarithmically smaller than 2, more precisely $ 0< \cH^{f}(B[0,1])<\infty$ where $f$ is the dimension function $f(r) = r^2\log \log(1/r)$ (for small $r$), see \cite{CT}. However, the projection $ \proj_{P}(B[0,1])$ of $B[0,1]$ onto any given plane $P$ has exactly the same distribution as a Brownian motion in the plane, which is almost surely of Hausdorff dimension 2, or precisely,  $ 0< \cH^{g}(\proj_{P}(B[0,1]))<\infty$ where $g$ is the dimension function $g(r) = r^2 \log(1/r) \log \log \log(1/r)$, see \cite{T64}.  This example, where the exact dimension functions of a set and of almost all its projections onto a plane can be identified, illustrates the sort of change in exact dimension that may occur under projection.
\medskip

\noindent{\em Sets with no exceptional projections. }  The dimension result \eqref{dimkaufman} for   the   set of  exceptional projections has been extended in various ways -- see \cite{falfrajin, Mat2} and references within.   We highlight a result  concerning sets  $A$ for which there are no exceptional projections; that is,  sets $A$ for which $E(A)  =  \emptyset$.
 \begin{thcs}
Let $A\subset \R^2$ be a self-similar set with dense rotations.  Then
~\vspace*{-0ex}
\begin{equation}  \label{dimps}
\dim {\proj}_{\theta}A=  \min \{ \dim A, 1 \}    \quad { for \   all  }  \quad \theta \in[0,\pi).
   \end{equation}
 \end{thcs}

 \noindent This theorem was proved by Peres and Shmerkin \cite{PS}  and subsequently generalized by  Hochman and  Shmerkin  \cite{HS}.  Now suppose  $A$ is self-similar set with dense rotations and  $f$ and $g$  are dimension functions as in Theorem \ref{mainthm}.  It is  natural to ask  whether or not
 the conclusion of part (ii) of Theorem \ref{mainthm} is actually valid for all $\theta$ rather than  just almost all $\theta \in[0,\pi)$.

% [COMMENT - KJF: Provided $A$ satisfies, say, the strong separation condition, $0< \cH^{s}(A)<\infty$ where $s =\dim A$. The interesting question is whether there is a dimension function $f$ such that $0< \cH^{f}({\proj}_{\theta}A)<\infty$ for all $\theta$, or at least a partition of dimension functions into classes for which for all $\theta$,     $\cH^{f}({\proj}_{\theta}A)= 0$ and $\cH^{f}({\proj}_{\theta}A)= \infty$. Possibly relevant are results of Eroglu and Farkas (see arxiv:1307.2841.pdf)   that, given dense rotations $\cH^{s}({\proj}_{\theta}A)= 0$ for all $\theta$ - perhaps the arguments could be refined in terms of dimension functions.]

\vspace{2ex}

\noindent{\em Lengths of projections. }
It is natural to seek a finer version of part (ii) of Marstrand's theorem  which gives a criterion for almost all projections of a set to have positive length. One aspect of this was investigated  by Peres and Solomyak \cite{PSol}, who considered dimension functions $f$ such that  $f(r)/r^2$ is decreasing for $r>0$ (a condition that holds in virtually all cases of interest).  The following statement constitutes parts (i) and (ii) of their main result \cite[Theorem 1.1]{PSol}.

 \begin{thbs}
Let $f$ be a dimension function  such that  $f(r)/r^2$ is decreasing. Then $\int_0^1 r^{-2}f(r) \mathrm{d} r <\infty$ if and only if for any Borel set $A\subset \R^2$ with $ \cH^{f}(A)>0$ one has that ${\mathcal L}(\proj_{\theta}A)>0$  for almost all $\theta \in[0,\pi)$.
 \end{thbs}

 \noindent Note that the integral convergence  condition in the above theorem  is exactly condition \eqref{(1)} in Theorem \ref{p3}  with  $g$ given by $g(r)= r$ and so $\cH^g $ is Lebesgue measure ${\mathcal L}$ .
\medskip

\vspace*{3ex}

\noindent{\bf Acknowledgements.}  SV would like to thank Julien Barral and St\'{e}phane Seuret for organising the wonderful conference  ``Fractals and Related Fields III''  in Porquerolles (19-25, September 2015).  It was during one of the many excellent talks at this conference that the problem of Marstrand for sets of dimension zero (and hence for general dimension functions) reared its ugly head in the sparse grey matter of SV's head!    SV and AZ  would like to thank Henna Koivusalo for many many useful conversations and patiently listening to our ramblings.

\noindent We would like to thank the referee for carefully reading the original manuscript. Her/his comments  have helped improve the clarity of the paper.

\vspace*{3ex}

%%%%%%%%%%%%%%%%%%%%%%%%%%%%%%%%%%%%%%

{\small

}

{\small

%%%%%%%%%*******************

%%%%%%%%%*******************

\vspace{5mm}

\noindent Victor Beresnevich: Department of Mathematics,
University of York,

\vspace{-2mm}

\noindent\phantom{Victor Beresnevich: }Heslington, York, YO10
5DD, England.

%\vspace{0mm}

\noindent\phantom{Victor Beresnevich: }e-mail: vb8@york.ac.uk

%%%%%%%%%*******************

%%%%%%%%%*******************

\vspace{5mm}

\noindent Kenneth Falconer: Mathematical Institute,
University of St Andrews,

\vspace{-2mm}

\noindent\phantom{Kenneth Falconer: }North Haugh, St Andrews, Fife,
KY16 9SS,
Scotland.

%\vspace{0mm}

\noindent\phantom{Kenneth Falconer: }e-mail: kjf@st-andrews.ac.uk

%%%%%%%%%*******************

%%%%%%%%%*******************

\vspace{5mm}

\noindent Sanju L. Velani: Department of Mathematics, University of York,

\vspace{-2mm}

\noindent\phantom{Sanju L. Velani: }Heslington, York, YO10 5DD, England.

%\vspace{0mm}

\noindent\phantom{Sanju L. Velani: }e-mail: slv3@york.ac.uk

%%%%%%%%%*******************

%%%%%%%%%*******************

\vspace{5mm}

\noindent Agamemnon Zafeiropoulos: Department of Mathematics,
University of York,

\vspace{-2mm}

\noindent\phantom{Agamemnon Zafeiropoulos: }Heslington, York, YO10
5DD, England.

%\vspace{0mm}

\noindent\phantom{Agamemnon Zafeiropoulos: }e-mail: az629@york.ac.uk

}

%%%%%%%%%%%%%%%%%%%%%%%%%

\newpage
%%%%%%%%%%%%%%%%%%%%%%%%%%%%%%%%%%%%%%%%%%%%%

\begin{appendix}

\begin{center} {\LARGE  Appendix:  \  The gap of uncertainty  \\ ~ \hspace*{10ex} }
\\ ~ \vspace*{-2ex} \\
{\large David Simmons \qquad   Han Yu    \qquad Agamemnon Zafeiropoulos }
\end{center}

\vspace*{6ex}

%\section{ On the necessity of condition \eqref{(1)} }

As discussed in Remark \ref{rem2} and explicitly demonstrated in  \S\ref{inad},  the integral convergence condition \eqref{(1)} gives rise to a gap of uncertainty. It is natural to ask whether this condition is really necessary. Namely, if $f$ and $g$ are dimension functions such that \eqref{(1)} is not satisfied, then does there exist a set $A \subseteq \R^2$ such that $\cH^f(A) > 0$ but $\cH^g(\proj_\theta A) = 0$ for almost all $0\leq \theta < \pi$? In this appendix we partially answer this question by providing sufficient conditions on $f$ and $g$ for the existence of such a set $A$. Our construction of this set is a generalization of a construction of Martin and Mattila  \cite{mma}, in which they proved that for every $0 < s \le  1$ there exists a set $A \subseteq \R^2$ such that $\cH^s(A) > 0$ but $\cH^s(\proj_\theta A) = 0$ for all $0\leq \theta < \pi$.   By making a careful quantitative analysis of their construction, we are able to improve their result and establish the following statement.

\vspace{2mm}

% A paper should not have two different theorems with the same name, so I changed the name of this theorem.
\begin{thdet}
\label{propositionpartialconverse}
Let $f,g$ be dimension functions such that $f$ is doubling with exponent $s_1\leq 1$ and codoubling with exponent $s_2>0$, and such that
\begin{equation} \label{fgrelation}
g(r) \leq M\ f\left(r\log(r^{-1})\right) \,\quad  \forall \,  0< r < r_0
\end{equation}
where $r_0 > 0 $ and $ M>0$ are  constants. Then there exists a set $A\subseteq \R^{2}$ with $0< \cH^{f}(A)< \infty $ but $\cH^g(\proj_{\theta}A) = 0$ for all $0\leq \theta < \pi$.
\end{thdet}

%By making a careful quantitative analysis of their construction, we are able to improve their result and establish the following result.
%
%
%
%\vspace{2mm}
%
%\begin{thdet}
%\label{propositionpartialconverse}
%Let $f,g$ be dimension functions such that $f$ is doubling with exponent $s_1\leq 1$ and codoubling with exponent $s_2>0$, and such that
%\begin{equation} \label{fgrelation}
%g(r) \leq M\ f\left(r\log(r^{-1})\right) \,
%\end{equation}
%where $M>0$ is a constant. Then there exists a set $A\subseteq \R^{2}$ with $\cH^{f}(A)>0$ but $\cH^g(\proj_{\theta}A) = 0$ for all $0\leq \theta < \pi$.
%\end{thdet}
%
%
%This is equivalent to Theorem A  appearing in Remark \ref{rem2}.  For convenience,  we have expressed doubling and codoubling constants in terms of exponents.  A dimension function $f$  is called {\it codoubling with exponent $s$} if there  exist constants $s >0 $, $ \kappa >0 $ and $r_1 > 0$ such that
%\begin{equation}  \label{codouble2}
%f(r \lambda)   \, \le   \, \kappa  \; \lambda^s f(r)    \qquad  \forall \  0 <  \lambda < 1  \ \ {\rm and } \ \ 0< r < r_1 \, .
%\end{equation}
% The notion of codoubling with constant $c > 1 $ as given by \eqref{codoublesv} is easily seen to be equivalent to the notion of codoubling with exponent  $s > 0 $  (cf.   Lemma \ref{double=}).

Recall, a dimension function $f$  is called {\it doubling with exponent $s$} if there exist constants $\kappa >0$ and $r_1 > 0$ such that
\begin{equation}  \label{double3}
f(r \lambda)   \, \ge   \, \kappa  \; \lambda^s f(r)    \qquad  \forall \  0 <  \lambda < 1  \ \ {\rm and } \ \ 0< r < r_1 \, .
\end{equation}
Moreover, $f$ is called {\it codoubling with exponent $s$} if there exist constants $\kappa >0$ and $r_1 > 0$ such that
\begin{equation}  \label{codouble2}
f(r \lambda)   \, \le   \, \kappa  \; \lambda^s f(r)    \qquad  \forall \  0 <  \lambda < 1  \ \ {\rm and } \ \ 0< r < r_1 \, .
\end{equation}
Note that these definitions of doubling and codoubling with exponent $s$ are different from the definitions of doubling and codoubling (with constant $c>1$) appearing in \S\ref{mainresultsec}.   It is easily seen that if a function is doubling (resp. codoubling) in the sense of \eqref{double} (resp. \eqref{codoublesv}) with constant $c > 1$, then it is also doubling (resp. codoubling) in the sense of \eqref{double3} (resp. \eqref{codouble2}) with exponent $s = \log_2c$. (See Lemma \ref{double=} for the case of doubling functions.) However, observe that if $f$ is doubling with exponent $s \leq 1$ then it is not necessarily doubling with constant $c \leq 2$. Also, if a function is  codoubling in the sense of \eqref{codouble2} (with exponent $s>0$) then it is not necessarily codoubling in the sense of \eqref{codoublesv} (with constant $c>1$). Thus Theorem A  is strictly stronger than Theorem A$'$ appearing in Remark \ref{rem2}.

\vspace{2mm}
\begin{rem}The assumption that the function $f$ is doubling with exponent $s\leq 1$ restricts our attention  to subsets  $A\subseteq \R^{2}$ of Hausdorff dimension at most $1$, which is expected given the nature of the problem  (cf. Remark \ref{ohyes}).
\end{rem}

%\vspace{2mm}
%\begin{rem}
%This result does not contradict the main theorem, since when $f$ and $g$ satisfy \eqref{fgrelation} we have
%\begin{eqnarray*}
%\int_{0}^{1}\frac{\mathrm{d}f(r)}{g(r)} &\gg & \int_{0}^{1}\frac{\mathrm{d}f(r)}{f\left( r\log (r^{-1})\right)} \\ [2ex]
%& =& \int_{0}^{1}\frac{\mathrm{d}r}{f^{-1}(r)\log \frac{1}{f^{-1}(r)}} \\ [2ex]
%&\gg & \int_{0}^{1} \frac{\mathrm{d}r}{r\log(r^{-1})} \\ [2ex]
%&=& +\infty,
%\end{eqnarray*}
%and $f,g$ do not satisfy the integral condition \eqref{(1)}.
%\end{rem}

\vspace{2mm}
\begin{rem}
\label{remlogp}
The growth condition  \eqref{fgrelation}    can be replaced by any condition of the form
\begin{equation}  \label{codouble22} g(r) \ll f\left(r\log(r^{-1})\log_3(r^{-1})\log_4(r^{-1})\cdots \log_p(r^{-1}) \right)
\end{equation}
where $p\geq 3$ is a positive integer and $r> 0 $ is sufficiently small -- see Remark \ref{remlogpa} below. Here we write $\log_2t= \log\log t, \ \log_3 t=\log\log\log t,$ \ etc  .
\end{rem}
\vspace{2mm}

\begin{rem}
\label{showingit}
Under the assumptions of Theorem A the integral
$$
\int_{0}^{1}\frac{\mathrm{d}f(r)}{g(r)}  \ =  \   \frac{f(1)}{g(1)}   \, - \, \int_{0}^{1}f(r)  \ \mathrm{d \!}\left(\frac{1}{g(r)}\right)   $$ diverges, hence  the integral convergence condition \eqref{(1)} is not satisfied and in turn Theorem~\ref{mainthm} is not violated. To see this, observe that since $f$ is doubling with exponent $s_1\leq 1$,   we have
\begin{eqnarray*}
f\left( r\log\frac{1}{r}\right) \ll f(r)  \, \log\frac{1}{r} \hspace{6mm} (r\to 0^+).
\end{eqnarray*}
Now since $f$ is codoubling with exponent $s_2>0$, we have
\[
f(r) \ll r^{s_2}
\]
and thus
\[
\log\frac{1}{r} \ll \log\frac{1}{f(r)}  \hspace{4mm} (r\to 0^+).
\]
Together the above estimates yield
\begin{eqnarray*}
f\left( r\log\frac{1}{r}\right) \ll f(r)\log\frac{1}{f(r)}  \, .
\end{eqnarray*}
It then follows that for any $g$ satisfying condition \eqref{fgrelation},  we have
\begin{eqnarray*}
 \int_{0}^{1}\frac{\mathrm{d}f(r)}{g(r)}   &\gg & \int_{0}^{1}\frac{\mathrm{d}f(r)}{f\left(r\log\frac{1}{r}\right)} \\[2ex]
&\gg & \int_0^1\frac{\mathrm{d}f(r)}{f(r)\log\frac{1}{f(r)}}  \\[2ex]
&=& \int_0^{f^{-1}(1)}\!\!\frac{\mathrm{d}x}{x\log\frac{1}{x}}  \ = \ \infty \, . \\[2ex]
\end{eqnarray*}
\end{rem}

\vspace{4mm}

\noindent{\bf \large{A.1 \ \  Construction of an $f$-set} }

\vspace{2mm}

Given a dimension function $f$, a set $A\subseteq \R^{n}$ is called an \emph{$f$-set} if $0<\cH^f(A)<\infty$. Here we present the construction of an $f$-set $A\subseteq \R^2$ for a given function $f$, which is similar to the one presented by Martin and Mattila in \cite[Section 5.3]{mma} for dimension functions of the form $r\mapsto r^s$. In the next section, we will show that by choosing the parameters of the construction appropriately, the resulting $f$-set $A$ will satisfy $\cH^g(\proj_\theta A) = 0$ for all $0\leq \theta < \pi$.\\

\noindent
Throughout, $(r_k)_{k=0}^{\infty}$ is a decreasing sequence of positive real numbers tending to $0$, $(N_k)_{k=1}^{\infty}$ is a sequence of positive integers $\geq 2$, and $(\theta_k)_{k=1}^{\infty}$ is a sequence of angles $0\leq \theta_k < \pi, \ k\geq 1$. The sequences $(r_k)_{k=0}^{\infty}$ and $(N_k)_{k=1}^{\infty}$ will be assumed to satisfy the inequalities
\begin{equation} \label{ab2}
a\leq N_1\cdots N_k f(r_k)\leq 2a
\end{equation}
and
\begin{equation}
\label{rk1rk}
N_{k+1} r_{k+1} < r_k
\end{equation}
for all $k\geq 0$, for some constant $a > 0$. \\

\noindent
Let $A_0$ be the closed disc of radius $r_0$ centered at the origin.  In the first step, inside $A_0$ we consider $N_1$ subdiscs of radius $r_1$, denoted $C_1,\ldots,C_{N_1}$ and defined as follows: their centers are equally spaced, lying on the diameter of $A_0$ which forms angle $\theta_1$ (measured counterclockwise) with the horizontal axis, and the boundaries of first and last subdisc are tangent to the boundary of $A_0$. Condition  \eqref{rk1rk} guarantees that these subdiscs are disjoint. Let $d_1 = \theta_1$, and set
$$ A_1 = \bigcup_{i=1}^{N_1}C_i.   $$
Now inductively assume that for some $k\geq 1$ we have defined the discs $C_{i_1\ldots i_k},\ 1\leq i_j\leq N_j,\ 1\leq j\leq k,$ each of radius $r_k$. \vspace{2mm}

\noindent
At the $(k+1)$st step, inside each disc $C_{i_1\ldots i_k}$ we consider $N_{k+1}$ subdiscs $C_{i_1\ldots i_k 1},\ldots, C_{i_1 \ldots i_k N_{k+1}}$, each of radius $r_{k+1}$, defined as follows: their centers are equally spaced along the diameter of $C_{i_1\ldots i_k}$ which forms angle $\theta_{k+1}$ with the line containing the centers of the discs of the $k$th step, and the boundaries of the first and last subdiscs are tangent to the boundary of $C_{i_1\ldots i_k}$. Again, condition  \eqref{rk1rk} guarantees that these subdiscs are disjoint. Let $d_{k+1}\equiv \theta_1+\ldots + \theta_{k+1} \ ( {\rm mod }\ \pi)$, so that for any disc $C_{i_1\ldots i_k}$ of $A_k$, $d_{k+1}$ is the angle between the diameter of $C_{i_1\ldots i_k}$ used to define the subdiscs of $C_{i_1\ldots i_k}$ and the horizontal axis. Set
$$ A_{k+1}= \mathop{\bigcup_{1\leq i_j \leq N_j }}_{(j=1,\ldots, k+1)}C_{i_1\ldots i_{k+1}}. $$
We complete the construction by setting
$$A= \bigcap_{k=1}^{\infty}A_k. $$
%\vspace{1mm}
\noindent
We show that under certain conditions on $f$ and appropriate choices of the sequences $(r_k)_{k=0}^\infty$ and $(N_k)_{k=1}^\infty$, the set $A$ is an $f$-set.
\vspace{2mm}

%\begin{definition}
%A function $f:(0,\infty)\to (0,\infty)$ is said to be {\it doubling with exponent $s > 0$} if there exists a constant $\kappa > 0$ such that for all $0 < \lambda < 1$ and $r > 0$ sufficiently small, we have
%\[
%f(\lambda r) \geq \kappa \lambda^s f(r).
%\]
%\end{definition}

%\begin{proposition}  \label{propositionconstruction}
%Let $f$ be a dimension function which is doubling with exponent $s\leq 1$, and let $(r_k)_{k = 0}^\infty$ be a sequence satisfying the inequalities
%\begin{equation} \label{rad1}
%f(r_{k+1}) < \frac{1}{4} f(r_k)
%\end{equation}
%and
%\begin{equation} \label{rad2}
%\frac{f(r_{k+1})}{r_{k+1}} > 3 \frac{f(r_k)}{r_k}
%\end{equation}
%for all $k\geq 0$. Let $(\theta_k)_{k = 1}^\infty$ be any sequence of real numbers. Then the parameter $a > 0$ and the sequence $(N_k)_{k = 1}^\infty$ can be chosen so as to satisfy \eqref{ab2} and \eqref{rk1rk} for all $k\geq 0$. The resulting set $A\subseteq\R^2$ constructed as above satisfies
%$$0\ <\ \cH^f(A)\ <\ \infty .$$
%\end{proposition}

\begin{thdeta}  \label{propositionconstruction}
Let $f$ be a dimension function which is doubling with exponent $s\leq 1$, and let $(r_k)_{k = 0}^\infty$ be a sequence satisfying the inequalities
\begin{equation} \label{rad1}
f(r_{k+1}) < \frac{1}{4} f(r_k)
\end{equation}
and
\begin{equation} \label{rad2}
\frac{f(r_{k+1})}{r_{k+1}} > 3 \frac{f(r_k)}{r_k}
\end{equation}
for all $k\geq 0$. Let $(\theta_k)_{k = 1}^\infty$ be any sequence of real numbers. Then the parameter $a > 0$ and the sequence $(N_k)_{k = 1}^\infty$ can be chosen so as to satisfy \eqref{ab2} and \eqref{rk1rk} for all $k\geq 0$. The resulting set $A\subseteq\R^2$ constructed as above is an $f$-set.
\end{thdeta}
\begin{proof}
Let $a = f(r_0)$, so that \eqref{ab2} automatically holds when $k = 0$. Now inductively assume that for some $k\geq 0$ we have chosen $N_1,\ldots, N_k \geq 2$ such that \eqref{ab2} holds. Since
\begin{eqnarray*}
\frac{2a}{N_1\cdots N_k f(r_{k+1})} - \frac{a}{N_1\cdots N_k f(r_{k+1})} &=& a\frac{f(r_k)}{f(r_{k+1})}\frac{1}{N_1\cdots N_k\ f(r_k)} \\[2ex]
& \stackrel{\eqref{ab2}}{\geq} & \frac 12 \frac{f(r_k)}{f(r_{k+1})} \\[2ex]
&\stackrel{\eqref{rad1}}{>} & 2,
\end{eqnarray*}
the interval
$$
\left[ \frac{2a}{N_1\cdots N_k f(r_{k+1})} \, ,   \frac{a}{N_1\cdots N_k f(r_{k+1})} \right]
$$
%whose endpoints are the two numbers on the left-hand side
contains a positive integer $N_{k+1}\geq 2$.  Thus, the inequality
\begin{equation}
\label{ab}
a\ \leq \ N_1  \cdots N_k N_{k+1}\hspace{1mm}f(r_{k+1})\ \leq \ 2a
\end{equation}
is satisfied. This completes the inductive step, thus demonstrating that the sequence $(N_k)_{k=1}^\infty$ can be chosen so that \eqref{ab2} holds for all $k \geq 0$.\\

\noindent
To demonstrate \eqref{rk1rk}, we note that
\begin{align}
\label{Nk1comparison}
N_{k+1} & \stackrel{\eqref{ab}}{\leq} \frac{2a}{N_1\cdots N_k f(r_{k+1})}
\stackrel{\eqref{ab2}}{\leq} 2\cdot \frac{f(r_k)}{f(r_{k+1})}
\stackrel{\eqref{rad2}}{<} \frac{2}{3}\cdot\frac{r_k}{r_{k+1}}
\end{align}
and in particular $N_{k+1} r_{k+1} < r_k$.\\

\noindent
For each $k\in\N$,  the set $A_k$ is a cover of $A$ consisting of discs of radius $r_k$. The number of balls in this cover is $ N_1 \cdots  N_k $, hence for $k\in\N$ we have
$$ \cH^f_{r_k}(A)\ \leq \ N_1\cdots N_k f(r_k)\ \leq \ 2a $$
and thus
$$ \cH^f(A) \ = \ \sup_{k>0}\cH^f_{r_k}(A)\ \leq \ 2a \ < \ \infty.  $$
Now consider the probability measure $\mu$ supported on $A$ which is defined by assigning each of the discs $C_{i_1\ldots i_k}$ of $A_k$ the same measure, i.e. by setting
$$ \mu(C_{i_1\ldots i_k}) = \frac{1}{N_1\cdots N_k}\cdot$$
We claim that for all $x\in A$ and $r>0$ small enough,
\begin{equation}
\label{massdist}
\mu\left( B(x,r) \right) \leq C f(r)
\end{equation}
for some constant $C>0$. By the Mass Distribution Principle (see for example \cite[Lemma~3]{mem}), this will imply that $\cH^f(A)>0$ and thus complete the proof.\\

\noindent
Whenever $r,r'$ are sufficiently small and $r < r'$, since $f$ is doubling with exponent $s\leq 1$, for some constant $\kappa > 0$ we have
\begin{align*}
\frac{f(r)}{r} & \geq \kappa \frac{1}{r}\left( \frac{r}{r'} \right)^{s}f(r')
= \kappa \left(\frac{r'}{r}\right)^{1-s}\frac{f(r')}{r'}
\geq \kappa\ \frac{f(r')}{r'},
\end{align*}
which gives
\begin{equation}
\label{frrdecreasing}
\frac{f(r')}{r'} \ \leq \ \kappa^{-1}\frac{f(r)}{r} \cdot
\end{equation}
Now fix $x\in A$ and $r>0$, and let $k\in\N$ be maximal such that $B(x,r)\cap A$ is contained in only one disc of $A_k$. Consider the following cases:

\noindent
{\bf Case 1:} \  $r<r_{k}.$ \ Let $s_{k+1}$ be the common distance between any two consecutive subdiscs of $A_{k+1}$. Then by subdividing the appropriate diameter of $C_{i_1\ldots i_k}$ into intervals consisting of its intersections with discs $C_{i_1\ldots i_k j}$ as well as the gaps between them, we find that
\begin{equation} \label{space}
2N_{k+1}r_{k+1} + (N_{k+1}-1)s_{k+1}=2r_k.
\end{equation}
Now in any sequence of $n$ consecutive subdiscs of $A_{k+1}$, the distance between the first and last subdiscs in this sequence is
\[
(n-1)s_{k+1} + (n-2)2r_{k+1} > (n-2) (s_{k+1} + 2r_{k+1}).
\]
Since the diameter of $B(x,r)$ is $2r$, if $n$ is the number of subdiscs of $A_{k+1}$ that intersect $B(x,r)$, then the distance given above must be less than $2r$. It follows that
\begin{equation}
\label{nbound1}
n \leq 2 + \frac{2r}{2r_{k+1} + s_{k+1}}\cdot
\end{equation}
On the other hand, we have
\begin{align*}
N_{k+1}(s_{k+1}-r_{k+1}) & \; >  \;  (N_{k+1}-1)s_{k+1} - N_{k+1}r_{k+1} \\
& \stackrel{\eqref{space}}{=} 2r_k - 3N_{k+1}r_{k+1}   \\  \;
& \stackrel{\eqref{Nk1comparison}}{>}  \;  0,
\end{align*}
which implies that
\begin{equation} \label{zx}
s_{k+1}>r_{k+1}  \, .
\end{equation} On the other hand, by the maximality of $k$, $B(x,r)$ intersects at least $2$ discs of $A_{k+1}$, including the disc containing $x$, and thus it follows that  $r>s_{k+1}$. This together with \eqref{zx} implies that  $r>r_{k+1}$ and so
\[
\frac{2r}{2r_{k+1} + s_{k+1}} \, \geq  \,  \frac{2}{3} \; .
\]
Hence,
\begin{align*}
n &     \stackrel{\eqref{nbound1}}{\leq}   \;  4\left(\frac{2r}{2r_{k+1} + s_{k+1}}\right)  \\[2ex]
&  \stackrel{\eqref{space}}{=}   \; 8r \left( 2r_{k+1}+2\frac{r_k-N_{k+1}r_{k+1}}{N_{k+1}-1} \right)^{-1}   \\[2ex]
& \leq   \; 4r\left( r_{k+1} + \frac{r_k-N_{k+1}r_{k+1}}{N_{k+1}}\right)^{-1} \\[2ex]
& =   \; 4\frac{N_{k+1}}{r_k} r \\[2ex]
&   \stackrel{\eqref{Nk1comparison}}{\leq}  \;    8\frac{f(r_k)}{f(r_{k+1})}\cdot \frac{r}{r_k} \\[2ex]
&  \stackrel{\eqref{frrdecreasing}}{\leq}  \;   \frac{8}{\kappa}\ \frac{f(r)}{f(r_{k+1})}
\end{align*}
\noindent
Since each subdisc of $A_{k+1}$ has measure $\frac{1}{N_1\cdots N_{k+1}}$, it follows that
$$ \mu\left( B(x,r)\right)\leq \frac{8}{\kappa}\ \frac{f(r)}{f(r_{k+1})}\cdot \frac{1}{N_1\cdots  N_{k+1}} \stackrel{\eqref{ab}}{\leq} \ \frac{8}{a \kappa}\ f(r).$$
{\bf Case 2:} \  $r\geq r_{k}.$ \ Let $C_{i_1\ldots i_k}$ be the unique disc of $A_k$ intersecting $B(x,r)$, which exists by the definition of $k$. Then
\[
\mu(B(x,r)) \leq \mu(C_{i_1\ldots i_k} )
= \frac{1}{N_1\cdots N_k}
\stackrel{\eqref{ab2}}{\leq} \frac{1}{a}f(r_k)
\leq \frac{1}{a}f(r),
\]
where in the last inequality, we have used the fact that $f$ is increasing.\\

\noindent
Thus in either case, \eqref{massdist} holds with $C = \max  \left\{ \frac{8}{a\kappa},\frac{1}{a}\right\} >0$.
\end{proof}
\vspace{2mm}

\begin{rem}
Note that Proposition A applies to any possible sequence of angles $(\theta_k)_{k=1}^{\infty}$, indicating that varying the sequence of angles may cause the quantity $\cH^f(A)$ to change slightly but will not affect the fact that it is finite and positive. The role of the sequence $(\theta_k)_{k=1}^{\infty}$ will become apparent in  the next section.
\end{rem}

%\subsection{An $f$-set with projections of zero $\cH^{g}$-measure }

\vspace{4mm}

\noindent{\bf \large{A.2 \ \  Proof of  Theorem A }}

\vspace{2mm}

We show that if  $g$ satisfies  the growth condition  \eqref{fgrelation}
relative to $f$, the sequences $(r_k)_{k=0}^{\infty},\ (N_k)_{k=1}^{\infty}$ and $(\theta_k)_{k=1}^{\infty}$ in the aforementioned construction can be suitably selected so that the corresponding $f$-set $A\subseteq \R^2$  satisfies $\cH^g(\proj_{\theta}A) = 0$ for all $0\leq \theta< \pi$.

\vspace{2mm}

First, we claim that the sequence $(r_k')_{k=k_0}^{\infty}$ defined by the formula
\begin{equation} \label{eqradii}
r_k' = (k\log k \log\log k)^{-k}
\end{equation}
satisfies \eqref{rad1} and \eqref{rad2} for all sufficiently large $k$. To prove this, we first observe that
\begin{equation} \label{thetarate}
\frac{r_{k+1}'}{r_k'}\ \asymp \ \frac{1}{k\log k\log\log k} \ \to \ 0   \qquad {\rm as } \hspace{7mm} k\to\infty.
\end{equation}
On the other hand, by the doubling and codoubling hypotheses imposed on $f$,  there exist constants $\kappa_1, \kappa_2 >0 $ such that
$$ \kappa_1 \lambda^{s_1} f(r) \leq f(\lambda r) \leq \kappa_2\lambda^{s_2} f(r) $$
for all $0<\lambda < 1$ and $r>0$ sufficiently small. Since $r_k' > r_{k+1}'$ for all $k$ sufficiently large, we have that
$$ f(r_{k+1}') \leq \kappa_2 \left( \frac{r_{k+1}'}{r_k'}\right)^{s_2} f(r_k) $$
and
\begin{eqnarray*}
\frac{ f(r_{k+1}')}{r_{k+1}'} &\geq & \kappa_1\frac{1}{r_{k+1}'}\left( \frac{r_{k+1}'}{r_k'}\right)^{s_1} f(r_k') \\[1ex]
& = & \kappa_1  \left( \frac{r_{k+1}'}{r_k'}\right)^{s_1-1} \frac{ f(r_k')}{r_k'}.
\end{eqnarray*}
Thus by \eqref{thetarate}, the inequalities $\eqref{rad1}_{r_k = r_k'}$ and $\eqref{rad2}_{r_k = r_k'}$ are satisfied for all $k$ large enough. Let $k_1 \geq k_0$ be chosen so that $\eqref{rad1}_{r_k = r_k'}$ and $\eqref{rad2}_{r_k = r_k'}$ are satisfied for all $k\geq k_1$.\\

\noindent
Now consider the sequence $(r_k)_{k=0}^\infty$ defined by the formula
\[
r_k = r_{k + k_1}',
\]
and note that \eqref{rad1} and \eqref{rad2} are satisfied for all $k \geq 0$. Thus by Proposition A, we can choose a sequence $(N_k)_{k=1}^{\infty}$ such that \eqref{ab2} and \eqref{rk1rk} hold for all $k\geq 0$. Also note that by \eqref{thetarate}  we have that
\begin{equation}
\label{thetarate2}
\frac{r_{k+1}}{r_k} \asymp \frac{1}{k\log k\log\log k}\cdot
\end{equation}
Let the sequence of angles be defined by
$$ \theta_{k+1}= \frac{r_{k+1}}{r_k}, \hspace{7mm} k\geq 0.$$
Then
$$\sum_{k=1}^{\infty}\theta_{k}=\infty.\vspace{2mm} $$
Take an arbitrary $0\leq \theta < \pi$. Let $d_{\theta}$ denote the direction perpendicular to $L_{\theta}$, i.e. the direction of projection, $d_\theta \equiv \theta + \pi/2 \; \text{(mod $\pi$)}$. Since the series $\sum_{k=1}^{\infty}\theta_k$ diverges, there are infinitely many $k\in\N$ such that $d_{\theta}$ lies between $d_k$ and $d_{k+1}$.  For each of these values of $k$, the angle between $d_\theta$ and $d_{k+1}$ is at most $\theta_{k+1}$, and thus for each disc $C_{i_1\ldots i_k}$ of $A_k$, the distances from the centers of all subdisc $C_{i_1\ldots i_k j}$ of $C_{i_1\ldots i_k}$ from the diameter of $C_{i_1\ldots i_k}$ in the direction $d_{\theta}$ are at most
$$ r_k\sin \theta_{k+1} \leq r_k \theta_{k+1} =r_{k+1}.$$
This is because all of these centers lie on the diameter of $C_{i_1\ldots i_k}$ in the direction $d_{k+1}$. This means that within each disc of $A_k$, when we project the union of the subdiscs of $A_{k+1}$ onto $L_{\theta}$ we get an interval of length at most $4r_{k+1}$, which we can think of as the union of at most $4$ intervals of length at most $r_{k+1}$. The number of such intervals is equal to the number of discs of $A_k$, that is, $N_1\cdots N_k \leq  \frac{2a}{f(r_k)}$.\\

\noindent
We have shown that for infinitely many values of $k$ there is a cover of $\proj_{\theta}A$ which consists of $4N_1\cdots N_k$ intervals of length at most $r_{k+1}$, hence for such $k$ we obtain that
\begin{eqnarray} \label{cost}
\cH^{g}_{r_{k+1}}( \proj_{\theta}A ) & \leq  &   8a\frac{g(r_{k+1})}{f(r_k)}\   \nonumber \\[2ex]
 & \leq   &  8a M \frac{1}{f(r_k)} f\left(r_{k+1}\log(r_{k+1}^{-1})\log\log(r_{k+1}^{-1})\right).
\end{eqnarray}
Now \eqref{eqradii} implies that
\begin{align*}
\log(r_k^{-1}) &\asymp \log(r_k'^{-1}) \asymp k\log( k\log k\log\log k) \asymp k\log k, \\[2ex]
\log\log(r_k^{-1}) &\asymp \log\log(r_k'^{-1}) \asymp \log(k\log k) \asymp \log k
\end{align*}
as $k\to\infty$. Combining this estimate with \eqref{thetarate2}, \eqref{cost}, and the fact that $f$ is codoubling shows that
$$ \cH^{g}_{r_{k+1}}( \proj_{\theta}A )\leq \frac{M_1}{(\log k)^{s_2}} \, ,$$
where $M_1>0$ is some absolute constant. This implies that
$$ \cH^g(\proj_{\theta}A) = \lim_{k\to\infty} \cH^{g}_{r_{k+1}}( \proj_{\theta}A ) = 0$$
and thereby completes the proof of Theorem A.

%\vspace{2mm}
%\begin{rem}
%This result does not contradict the main theorem, since when $f$ and $g$ satisfy \eqref{fgrelation} we have
%\begin{eqnarray*}
%\int_{0}^{1}\frac{\mathrm{d}f(r)}{g(r)} &\gg & \int_{0}^{1}\frac{\mathrm{d}f(r)}{f\left( r\log (r^{-1})\right)} \\ [2ex]
%& =& \int_{0}^{1}\frac{\mathrm{d}r}{f^{-1}(r)\log \frac{1}{f^{-1}(r)}} \\ [2ex]
%&\gg & \int_{0}^{1} \frac{\mathrm{d}r}{r\log(r^{-1})} \\ [2ex]
%&=& +\infty,
%\end{eqnarray*}
%and $f,g$ do not satisfy the integral condition \eqref{(1)}.
%\end{rem}
%\vspace{2mm}
%\begin{rem}
%\label{remlogp}
%The growth condition  \eqref{fgrelation} in Theorem A   can be replaced by any condition of the form
%$$ g(r) \ll f\left(r\log(r^{-1})\log_3(r^{-1})\log_4(r^{-1})\cdots \log_p(r^{-1}) \right), \hspace{6mm} r\to 0^+$$
%where $p\geq 3$ is a positive integer. Here we write $\log_2t= \log\log t, \ \log_3 t=\log\log\log t,$ \ etc. The corresponding set in that case is constructed using the sequence $(r_k')_{k\geq k_0}$ defined by the formula
%$$ r_k' = \left( k\log k \log_2 k \cdots \log_p k \right)^{-k}.   $$  The proof is nearly identical and we leave it to the interested reader.
%\end{rem}
%
%

\vspace{2mm}
\begin{rem}
\label{remlogpa}
As mentioned in Remark \ref{remlogp}, the growth condition  \eqref{fgrelation} in Theorem A  can be replaced by any condition of the form \eqref{codouble22}.
The corresponding set in that case is constructed using the sequence $(r_k')_{k\geq k_0}$ defined by the formula
$$ r_k' = \left( k\log k \log_2 k \cdots \log_p k \right)^{-k}.   $$  The proof is nearly identical and we leave it to the interested reader.
\end{rem}

%%%%%%%%%%%%%%%%%%%%%%%%%%%%%%

\vspace{7mm}

\noindent David Simmons: Department of Mathematics,
University of York,

\vspace{-2mm}

\noindent\phantom{David Simmons: }Heslington, York, YO10
5DD, England.

%\vspace{0mm}

\noindent\phantom{David Simmons: }e-mail: david.simmons@york.ac.uk

%%%%%%%%%*******************

%%%%%%%%%*******************

\vspace{5mm}

\noindent Han Yu: Mathematical Institute,
University of St Andrews,

\vspace{-2mm}

\noindent\phantom{Han Yu: }North Haugh, St Andrews, Fife,
KY16 9SS,
Scotland.

%\vspace{0mm}

\noindent\phantom{Han Yu: }e-mail: hy25@st-andrews.ac.uk

%%%%%%%%%*******************

%%%%%%%%%*******************

\vspace{5mm}

\noindent Agamemnon Zafeiropoulos: Department of Mathematics,
University of York,

\vspace{-2mm}

\noindent\phantom{Agamemnon Zafeiropoulos: }Heslington, York, YO10
5DD, England.

%\vspace{0mm}

\noindent\phantom{Agamemnon Zafeiropoulos: }e-mail: az629@york.ac.uk

\end{appendix}


\begin{thebibliography}{99}



\bibitem{Beresnevich-Bernik-Dodson-Velani-Roth}
{V.~Beresnevich, V.~Bernik, M.~Dodson and  S.~L.~Velani},
Classical metric Diophantine approximation revisited, in  {\em   Analytic Number Theory.  Essays in Honour of Klaus Roth}. Eds. W. Chen, T. Gowers, H. Halberstam, W.M. Schmidt and R.C. Vaughan,  pp.~38--61, Cambridge University Press, 2009.


%\bibitem{memoirs} V. Beresnevich, H. Dickinson and
%S.L. Velani, {\em Measure theoretic laws for lim sup sets}. Mem. Amer. Math. Soc. {\bf 179} (2006) no. 846, 1--91.
%
%
%
%\bibitem{durham} V. Beresnevich, F. Ram\'{\i}rez and
%S.L. Velani, {\em
%Classical metric Diophantine approximation revisited},
%in Dynamics and Analytic Number Theory, Editors: Dmitry Badziahin, Alex Gorodnik, and Norbert Peyerimhoff. LMS Lecture Note Series 437, Cambridge University Press, (2016).   Preprint: arXiv:1601.01948.
%
%\bibitem{BV06}
%V.~Beresnevich, S.~L. Velani, \emph{A {M}ass {T}ransference
%{P}rinciple and
%  the {D}uffin-{S}chaeffer conjecture for {H}ausdorff measures}, Ann. \ Math.
%  \textbf{164} (2006), 971--992.

\bibitem{HDSV97}
H.~Dickinson and S.~L. Velani, Hausdorff measure and linear
forms, {\em J.\
  Reine Angew.\ Math.}   490 (1997), 1--36.

\bibitem{CT}
Z. Ciesielsk and S. J.  Taylor,
First passage times and sojourn times for Brownian motion in space and the exact Hausdorff measure of the sample path,
{\em Trans. Amer. Math. Soc.} 103 (1962), 434--450.

\bibitem{Dav}
R. O. Davis, Sets which are null or non-sigma-finite for every translation-invariant measure, {\em Mathematika} 18 (1971), 161--162.

\bibitem{F1}
K. J. Falconer, {\em The Geometry of Fractal Sets}, Cambridge University Press, 1985.

\bibitem{F}
K. J. Falconer, {\em Fractal Geometry: Mathematical Foundations and
Applications}, 3rd. Ed. John Wiley, 2014.

\bibitem{falfrajin}
K. J. Falconer,  J. Fraser and  X.  Jin,
Sixty years of fractal projections, in {\em Fractal Geometry and Stochastics V},  Eds. C. Bandt, K. J. Falconer and  M. Zahle, pp.~3--25,  Birkhauser, 2015.

\bibitem{HS}
M. Hochman and P. Shmerkin, Local entropy averages and projections of fractal measures, {\em Ann. of
Math.}(2) 175 (2012), 1001--1059.

\bibitem{K}
R. Kaufman, On the Hausdorff dimension of projections, {\em Mathematika} 15 (1968), 153--155.

\bibitem{Mar}
J. Marstrand, Some fundamental geometrical properties of plane sets of fractional dimension, {\em Proc. London Math. Soc.}(3) 4 (1954), 257--302.

\bibitem{Mat}
P. Mattila, {\em  Geometry of Sets and Measures in Euclidean Space}, Cambridge University Press, 1995.


\bibitem{Mat2}
P. Mattila, {\em  Fourier Analysis and Hausdorff Dimension, }Cambridge University Press, 2015.

\bibitem{mm}
M.A. Martin and P. Mattila,
{\em $k$-dimensional regularity classifications for $s$-fractals,} Trans. Amer. Math. Soc. {\bf 305} (1988), 293--315.

\bibitem{PS}  Y. Peres and P. Shmerkin, Resonance between Cantor sets, {\em Ergodic Theory Dynam. Systems} 29 (2009),  201--221.

\bibitem{PSol}  Y. Peres and P. Solomyak, The sharp Hausdorff measure condition for length of projections, {\em Proc. Amer. Math. Soc.} 133 (2005),  3371--3379.

\bibitem{Rog}
C. A. Rogers, {\em  Hausdorff Measures}, 2nd Ed. Cambridge University Press, 1998.

\bibitem{T}
S. J. Taylor,  On the connexion between Hausdorff measures and generalized capacity, {\em Proc. Cambridge Philos. Soc.} 57  (1961), 524--531.

\bibitem{T64}
S. J. Taylor,  The exact Hausdorff measure of the sample path for planar Brownian motion,
{\em Proc. Cambridge Philos. Soc.} 60 (1964), 253--258.




\end{thebibliography}

{\small \begin{thebibliography}{99}


\bibitem{mem}
V. Beresnevich, D. Dickinson, S. Velani, {\em Measure theoretic laws for limsup sets, } Memoirs of the American Mathematical Society {\bf 179 }(2006), no. 846, 1-91.

%\bibitem{F}
%K.J. Falconer, {\em Fractal Geometry: Mathematical Foundations and Applications.} John Wiley \& Sons, 1990.


\bibitem{mma}
M. Martin and P. Mattila,
{\em $k$-dimensional regularity classifications for $s$-fractals,} Trans. Amer. Math. Soc. {\bf 305} (1998), 293-315.


%\bibitem{Mata}
%P. Mattila, {\em  Geometry of Sets and Measures in Euclidean Space}, Cambridge University Press, 1995.
%



\end{thebibliography}

}
\end{document}